\documentclass[10pt]{article}
\usepackage{times}
\usepackage{amsmath,amsfonts,amstext,amssymb,amsbsy,amsopn,amsthm,eucal,slashed,amstext,enumerate,accents}
\usepackage{color}
\usepackage{txfonts}
\usepackage{dsfont}
\usepackage{hyperref}
\usepackage{graphicx}   

\numberwithin{equation}{section}
  \theoremstyle{plain}
 \newtheorem{theorem}{Theorem}
 
\newtheorem{proposition}{Proposition}
 \newtheorem{lemma}{Lemma}
 \newtheorem{corollary}{Corollary}
\newtheorem*{conjecture}{Conjecture}

 \theoremstyle{remark}

 \newtheorem{remark}{Remark}

\theoremstyle{definition}

\renewcommand{\div}{{\rm div}}

\newcommand{\pa}{\partial}

\newcommand{\Sn}{\mathbb{S}^n}
\newcommand{\Sp}{\mathbb{S}}
\newcommand{\N}{\mathbb{N}}

\newcommand{\Lap}{\Delta}

\newcommand{\RR}{\mathbb{R}}
\newcommand{\CC}{\mathbb{C}}

\allowdisplaybreaks

\title{Green functions for GJMS operators on spheres, Gegenbauer polynomials and \\rigidity theorems}
\author{Xuezhang Chen\thanks{X. Chen:xuezhangchen@nju.edu.cn.}~ and Yalong Shi\thanks{Y. Shi:shiyl@nju.edu.cn.}\\
		{\small $^{\ast}$$^{\dag}$School of Mathematics \& IMS, Nanjing University, Nanjing 210093, P. R. China }}
\date{}

\begin{document}
\maketitle

\begin{abstract}
We derive explicit representation formulae of Green functions for GJMS operators on  $n$-spheres, including the fractional ones. These formulae have natural geometric interpretations concerning the extrinsic geometry of the round sphere. Conversely, we discover that this special feature uniquely characterizes spheres among closed embedded hypersurfaces in $\RR^{n+1}$. Furthermore, for $n=3,4,5$ we prove a strong rigidity theorem for Green functions of hypersurfaces in $\RR^{n+1}$ using the Positive Mass Theorem.

\medskip

{\bf MSC2020: } 35J08, 53C24 (58J60, 53C21)
\end{abstract}
\medskip
{\footnotesize \tableofcontents}

\section{Introduction}\label{sec:Intro}

Let $(M^n,g)$ be a closed orientable smooth Riemannian manifold with total volume $V$. The Green function $G(P,Q)$  of the Laplacian operator on $M$ is a function $G(P,Q)$ on $M\times M$ satisfying 
$$-\Delta_{g,P} G(P,Q)=\delta_Q(P)-\frac{1}{V}$$
in the distribution sense. The existence and uniqueness (up to a constant) of Green function can be proved by Hadamard's parametrix method, see for example Aubin \cite[Chapter 4]{aubin_book}. For $n$-dimensional closed manifolds of positive Yamabe constant and $n \geq 3$, under conformal normal coordinates the expansion of Green function for the conformal Laplacian $P_2^g=-\Delta_g+\frac{n-2}{4(n-1)}R_g$  had been explored in Lee-Parker \cite{Lee-Parker}, where $R_g$ is the scalar curvature.

The GJMS operator $P_{2k}^g$ for $k \in \mathbb{Z}_+$ is a conformally covariant scalar differential operator whose principal part agrees with $(-\Delta_g)^k$, generalizing the conformal Laplacian. See Graham, Jenne, Mason and Sparling \cite{GJMS}. These operators are defined for any $n$-manifold when $2k\leq n$ or $2k>n$ and $n$ is odd. The operator $P_4^g$ was discovered earlier by Paneitz \cite{Paneitz} in $1983$, and hence also referred to as `Paneitz operator'.  See also \eqref{def:Paneitz_operator} for the definition. For even $n$ and $2k>n$, it is known that in general there is no conformally covariant differential operator with principal part $(-\Delta_g)^{k}$. However, the operator $P_{2k}^g$ does exist when $g$ is Einstein or locally conformally flat. The importance of GJMS operator and its  intimately associated $Q$-curvature was emphasized by Branson. 

For $\sigma \in (0,\frac{n}{2}]$, given a Poincar\'e-Einstein manifold $(X^{n+1},M^n, g_+)$  and a representative $g$ in the conformal infinity $(M,[g])$, Graham and Zworski \cite{Graham-Zworski} used the scattering operator (or the generalized Dirichlet-to-Neumann map) to define the \emph{fractional} GJMS operator $P_{2\sigma}^g$ on $(M,g)$ with principal part $(-\Delta_g)^\sigma$. Moreover, $P_{2\sigma}^g$ is formally self-adjoint and has the conformal covariance property that given $\hat g=e^{2u}g$, $P_{2\sigma}^g(e^{\frac{n-2\sigma}{2}u}\varphi)=e^{\frac{n+2\sigma}{2}u} P_{2\sigma}^{\hat g}(\varphi), \forall~ \varphi \in C^\infty(M)$. In particular, when $\sigma \in \mathbb{Z}_+$, $P_{2\sigma}^g$ recovers the GJMS operator constructed in \cite{GJMS}. Readers are also referred to Ache-Chang \cite[Section 2.1]{Ache-Chang} for details. Moreover, on spheres these conformally covariant operators $P_{2\sigma}^{\Sn}$ with respect to the round metric $g_{\Sn}$ were defined earlier by Branson \cite[Theorem 2.8]{Branson95}, as intertwining operators from the viewpoint of representation theory, in the most general case for $\sigma\in \mathbb{C}$ with $-\sigma\notin \frac{n}{2}+\mathbb{N}$; an alternative approach to derive such conformally covariant operators was implicitly covered by Fefferman-Graham \cite{Fefferman-Graham}.

From now on, we use the GJMS operator $P_{2\gamma}^g$  to incorporate the ordinary  and fractional GJMS operators. 

On the unit sphere $\Sn$, the GJMS operator in critical dimension $n$  is
$$
P_n:=P_n^{\Sn}=\begin{cases}
\displaystyle  \prod_{j=0}^{(n-2)/2}\Big(-\Delta_{\Sn}+j(n-1-j)\Big) &\qquad \mathrm{~~for~~} n \mathrm{~~even};\\
\displaystyle  \sqrt{-\Delta_{\Sn}+(\frac{n-1}{2})^2} \prod_{j=0}^{(n-3)/2}\Big(-\Delta_{\Sn}+j(n-1-j)\Big) &\qquad \mathrm{~~for~~} n \mathrm{~~odd}.
\end{cases}
$$
See Branson \cite[p.231]{Branson} and Beckner \cite{Beckner}. 
For $n \neq 2k$ and $k \in \mathbb{Z}_+$, the GJMS operator of order $2k$ is
$$
 P_{2k}:=P_{2k}^{\Sn}=\prod_{i=1}^k\left(-\Delta_{\Sn}+\big(\frac{n}{2}+i-1\big)\big(\frac{n}{2}-i\big)\right).
$$
 Especially, $P_2$ is the conformal Laplacian and $P_4$ is the Paneitz operator.
 See Fefferman-Graham \cite{Fefferman-Graham} and Gover \cite{Gover} for the factorization of GJMS operators.
 
The operators $P_n$ in critical dimension $n$ and $P_{2k}$ in dimension $n\neq 2k$ are conformally covariant. It is well-known that $P_n$ has spherical harmonics as eigenfunctions with the corresponding eigenvalues 
$$\lambda_l:=\lambda_l(P_n)=\frac{\Gamma(l+n)}{\Gamma(l)},  \qquad l \in  \N $$
and the  $Q$-curvature of the round metric is $Q_{\Sn}=(n-1)!$.

The Green function of GJMS operators on $\Sn$ can be similarly defined as follows:
\begin{equation}\label{Def:Green_fcn_critical}
P_{n} G(\cdot,Q)=\delta_Q-\frac{1}{|\Sn|}\qquad \mathrm{~~in~~critical~~dimension~~} n
\end{equation}
and
\begin{equation}\label{Def:Green_fcn_others}
P_{2k} G(\cdot,Q)=\delta_Q\qquad \mathrm{~~for~~ either~~} 2k<n \mathrm{~~or~~}\ 2k>n \mathrm{~~with~~} n \mathrm{~~odd}
\end{equation}
in the distribution sense. Equivalently, 
for all $u\in C^\infty(\Sn)$ there hold
     \begin{equation*}
     u(Q)=\frac{1}{|\Sn|}\int_{\Sn}u dV_{\Sn}+\int_{\Sn} G(\cdot, Q) P_nu dV_{\Sn}   
    \end{equation*}
    in the former case, and
    \begin{equation*}
     u(Q)=\int_{\Sn} G(\cdot, Q) P_{2k}u dV_{\Sn}  
    \end{equation*}
    in the latter case.

The case $2k>n$ and $n$ is even is exceptional. In this case, rewrite
$$P_{2k}=\prod_{i=\frac{n}{2}+1}^k\Big(-\Delta_{\Sn}-\big(\frac{n}{2}+i-1\big)\big(i-\frac{n}{2}\big)\Big)\prod_{i=1}^{\frac{n}{2}}\Big(-\Delta_{\Sn}+\big(\frac{n}{2}+i-1\big)\big(\frac{n}{2}-i\big)\Big)$$
and for each integer $i \in (\frac{n}{2},k]$, $(i-\frac{n}{2})(i+\frac{n}{2}-1)$ is an eigenvalue of $-\Delta_{\Sn}$, hence the kernel of $P_{2k}$ consists of functions other than constants, that is, $\{\mathrm{constants}\} \subsetneq \mathrm{ker}(P_{2k})$. Consequently, there \emph{can not exist} Green functions of the above two types. Of course one can define an appropriate `Green function' in this setting modulo the kernel of $P_{2k}$, see for example the proof of Hodge decomposition theorem \cite[Chapter 0]{GH}, and study it using the methods of the present work. We leave the details to interested readers.

The Green function of conformal Laplacian on $\Sn$ under stereographical projection is well known, which follows from Lemma 3.2 in \cite[Chapter 6]{SchYau}. In the initial Chinese version of \cite{SchYau}, one can also find an explicit formula there. The Green function of the Paneitz operator $P_4$ on non-critical dimensional spheres was available in Hang-Yang \cite{Hang-Yang, Hang-Yang16}. The Green function of $P_{2k}$ on odd dimensional spheres $\Sn$ once appeared in  \cite[formula (2.1)]{Hang}. For $\sigma\in (0,\frac{n}{2})$ and $\sigma\notin\mathbb{N}$, the Green function of $P_{2\sigma}$ was essentially derived by J. Qing and D. Raske using conformal covariance in \cite[Theorem 2.1]{QingRaske}. The other cases are probably known to some experts. However, to the best of the authors' knowledge, they never appear in the literature.

\medskip

The first purpose of this paper is to derive explicit representation formulae for Green functions of both local and non-local GJMS operators on spheres in a unified way. It is interesting to discover that our explicit representation formula has a particular geometric meaning, whose importance seems to be ignored in large, even for the familiar $P_2$ and $P_4$ operators.  

Before presenting our results,  we denote by $\|P-Q\|$  the \emph{extrinsic} distance between $P$ and $Q$, where $P,Q\in \Sn\subset\RR^{n+1}$.
\begin{theorem}\label{thm:GreenSn}
For $n\geq 2$ and $P,Q \in \Sn$, the following are the Green function $G(P,Q)$ of GJMS operator $P_{2\gamma}$ with respect to the round metric:
\begin{itemize}
\item[(1)] In critical dimension $n$, the Green function of $P_n$ is $G(P,Q)=-c_n\log\|P-Q\|+c$, where $$c_n=\frac{2}{(n-1)!|\Sn|}=\frac{1}{2^{n-1}\pi^{\frac{n}{2}}\Gamma(\frac{n}{2})}\qquad \mathrm{~~and~~}\quad  c \in \RR.$$  
\item[(2)] For $2k<n$, or $2k>n$ and $n$ is odd, the Green function of $P_{2k}$ is $G(P,Q)=c_{n,k}\|P-Q\|^{2k-n}$, where
$$c_{n,k}=\frac{\Gamma(\frac{n}{2}-k)}{2^{2k}\pi^{\frac{n}{2}}\Gamma(k)}.
$$
Here $\Gamma(z)$ is always understood as the unique meromorphic function on $\CC$ and coincides with $\int_0^\infty t^{z-1}e^{-t}dt$ when $z\in\RR_+$.
\item[(3)]For $\sigma\in (0,\frac{n}{2})\cup (\frac{n}{2},+\infty)$ and $\sigma\notin\frac{n}{2}+\mathbb{N}$, the Green function of $P_{2\sigma}$ is $G(P,Q)=c_{n,\sigma}\|P-Q\|^{2\sigma-n}$, where 
$$c_{n,\sigma}=\frac{\Gamma(\frac{n}{2}-\sigma)}{2^{2\sigma}\pi^{\frac{n}{2}}\Gamma(\sigma)}.$$
\end{itemize}
\end{theorem}

\begin{remark}
    \begin{enumerate}[(i)]
        \item Though case (2) of Theorem \ref{thm:GreenSn} is contained in case (3), we decide to write it separately in order to emphasize the ordinary GJMS operators. 
        \item These results should be compared with the fundamental solution of $(-\Delta_0)^k$ on $\RR^n$ (see, for example, \cite[p.44]{John}):
        $$\Gamma(x,y)=\begin{cases}
         \displaystyle c_{n,k}|x-y|^{2k-n} &\qquad \mathrm{~~for~~} 2k<n, \mathrm{ ~~or~~ } n \mathrm{~~ odd} ;\\ 
         \displaystyle \frac{(-1)^{k-\frac{n}{2}-1}|x-y|^{2k-n}\log|x-y|}{2^{2k-1}\pi^{\frac{n}{2}}\Gamma(k)(k-\frac{n}{2})!} &\qquad \mathrm{~~for~~} 2k\geq n \mathrm{~~and~~}  n \mathrm{~~even}.
        \end{cases}$$
        \item By conformal covariance, it is clear that for a round sphere with extrinsic radius $r>0$, the Green functions have the same form as that of the unit sphere.
        \item For a smooth finite quotient of $\Sn$, like lens space, the Green function of GJMS operators is related to the Green function of $\Sn$ by $G(\cdot,Q)\circ \imath=\sum_{\tilde Q\in \imath^{-1}(Q)}G_{\Sn}(\cdot,\tilde Q)$, where $\iota$ is the quotient map, as can be easily verified. See for example \cite[Lemma 2.4]{EscSch}.
    \end{enumerate}
\end{remark}

In all the previous work on Green functions for GJMS operators on spheres, the authors use the stereographic projection and conformal covariance of GJMS operators. For ordinary GJMS operators, it is easy to see that the pullback of the fundamental solutions of $(-\Delta_{\RR^n})^k$ satisfies the differential equations for Green functions. And one can check that it is indeed the Green function of the GJMS operator on the sphere by integrating against an arbitrary smooth test function. Readers are referred to Section\ref{sec:ordinary} for details.    

Since fractional GJMS operators are nonlocal,  one needs to be careful when reducing the problem to $\RR^n$ via stereographic projection. The previous method indeed works, see for example \cite{Branson95, Chang-Yang, QingRaske, Xu} etc. Instead, we adopt a new and unified approach, by using the follow fact from spectral theory: if $P$ is a self-adjoint elliptic (pseudo-differential) operator on a compact manifold $M$ with eigenvalues $\lambda_i\to\infty$ and corresponding $L^2$-orthonormal eigenfunctions $\varphi_i$, and if $\mathrm{Ker} P=\{0\}\ \text or\ \RR$, then the Green function of $P$ is
\begin{equation}\label{eqn:Green_spectrum}
 G(x,y)=\sum_{i:\lambda_i\neq 0}\frac{1}{\lambda_i}\varphi_i(x)\varphi_i(y).
\end{equation}
In our problem, the eigenfunctions of GJMS operators on spheres are precisely spherical harmonics. So the problem is to compute the series \eqref{eqn:Green_spectrum} consisting of spherical harmonics. Note that similar trick had been used by Beckner \cite{Beckner}. \\

The main tool for our evaluation of \eqref{eqn:Green_spectrum} is the Gegenbauer polynomial, with which, we can also give a rigorous proof of an \emph{astonishing} formula of GJMS operators acting on axially symmetric functions on spheres.
\begin{proposition}\label{prop:full_GHX}
For $n$ even, suppose a smooth function $u$ on $\Sn$ depends only on $x:=x_{n+1}$. Then there holds
\begin{equation}\label{full_formula:GHX}
P_n u(x)=(-1)^{\frac{n}{2}}[(1-x^2)^{\frac{n}{2}}u']^{(n-1)}.
\end{equation}
\end{proposition}

\begin{remark}
The above formula \eqref{full_formula:GHX}  in lower even dimensions was first proved by Gui-Hu-Xie \cite{Gui-Hu-Xie} via straightforward computations; however, a rigorous proof for all even $n$ does not appear before our work.
\end{remark}

As an application of this formula, we give a short and direct proof of the following uniqueness theorem for a higher order mean field equation on spheres, which first appeared in \cite[Theorem 1.1]{Gui-Hu-Xie} for all dimensions $n$.
\begin{theorem}\label{thm:meanfield}
   For $n$ even, suppose $u$ is an axially symmetric smooth solution of 
   \begin{equation}\label{eqn:MFE}
       \frac{1}{n+1} P_n u=(n-1)!(e^{nu}-1) \qquad \mathrm{~~on~~}\quad \Sp^n.
   \end{equation}
   Then $u=0$ on $\Sn$.
\end{theorem}

Since the Green functions of GJMS operators on $\Sn$ have such explicit geometric meanings, it is natural to ask whether such properties characterize the round sphere among all closed hypersurfaces in $\RR^{n+1}$. We are able to justify this for ordinary GJMS operators of order up to $4$.

\begin{theorem}\label{thm:rigidity}
  Let $(M,g)$ be a closed embedded hypersurface in $\RR^{n+1}$ with induced metric $g$. Assume one of the following conditions holds:
  \begin{itemize}
      \item[(1)] $n=2$, and the Green function of  $-\Delta_g$ has the form $G(P,Q)=-\frac{1}{2\pi}\log\|P-Q\|+C$ for any  $P,Q\in M$ and some $C\in \RR$;
      \item[(2)] $n\geq 3$, and the Green function for the conformal Laplacian $P_2^g$ has the form $G(P,Q)=c_{n,1}\|P-Q\|^{2-n}$ for any $P,Q\in M$;
      \item[(3)] $n\geq 3$ and $n\neq 4$, and the Green function for the Paneitz operator $P_4^g$ has the form $G(P,Q)=c_{n,2}\|P-Q\|^{4-n}$ for any $P,Q\in M$;
      \item[(4)] $n=4$ and the Green function for the Paneitz operator $P_4^g$ has the form $G(P,Q)=-\frac{1}{8\pi^2}\log\|P-Q\|+C$ for any $P,Q\in M$ and some $C \in \RR$.
  \end{itemize}
  Then $M$ is a round sphere. 
\end{theorem}

At this point, we would like to raise a stronger conjecture concerning the Green function rigidity problem:
\begin{conjecture}[Green function rigidity]
    Let $M^n\subset\RR^{n+1}$ be a closed hypersurface with induced Riemannian metric $g$. Suppose for some $k\in\mathbb{Z}_+$ the Green function $G$ of the GJMS operator $P_{2k}^g$ exists. Assume one of the following conditions holds: 
    \begin{enumerate}
        \item[(1)]$2k=n$ and  {\bf for some}  $Q\in M$, $G(\cdot, Q)$ is of the form $-c_n\log\|\cdot-Q\|+c$;
        \item[(2)]$2k<n$ or  $2k>n$ when $n$ is odd, and {\bf for some}  $Q\in M$, $G(\cdot, Q)$ is of the form $c_{n,k}\|\cdot-Q\|^{2k-n}$.
    \end{enumerate}
    Then $(M,g)$ is a round sphere.
\end{conjecture}

With the help of the Positive Mass Theorem, we confirm the conjecture for the conformal Laplacian in low dimensions:

\begin{theorem}\label{thm:rigidity_PMT}
  Let $(M,g)$ be a closed embedded hypersurface in $\RR^{n+1}$ where $n=3, 4$ or $5$, with induced metric $g$. Assume that there is a point $Q\in M$ such that the Green function for conformal Laplacian $P_2^g$ satisfies $G(P,Q)=c_{n,1}\|P-Q\|^{2-n}$ for any $P\in M$. Then $M$ is a round sphere. 
\end{theorem}

The restriction on $n$ in the above theorem comes from the following fact: Even though the scalar-flat manifold $(M\setminus\{Q\}, G(\cdot,Q)^{4/(n-2)}g)$ could be asymptotically flat of order $3$ after a further asymptotic coordinate change (see Remark \ref{rmk:coordinate change}), and hence the ADM mass is well defined when $3\leq n\leq 7$, we are able to show the vanishing of mass only when $n\leq 5$. 

Noting that $n=2$ is the critical dimension for conformal Laplacian, we can only confirm the conjecture under extra conditions. See Theorem \ref{thm:rigidity_analytic} in Section \ref{sec:strong rigidity}.

Although the above conjecture is only about the ordinary GJMS operators constructed in \cite{GJMS}, it might be true for the generic GJMS operator $P_{2\gamma}^g$ and deserves further exploration.

\medskip

In the next section, we shall present two warm-up examples: the derivation of Green functions for the usual Laplacian on $\Sp^2$ and for conformal Laplacian on $\Sn (n\geq 3)$. The results suggest that the pulling back of the fundamental solution for $(-\Delta_{\RR^n})^k$ on $\RR^n$ under stereographic projection should be the Green function for $P_{2k}$ on $\Sp^n$.  Then in Section \ref{sec:ordinary}, we verify that this is indeed the case by integrating against test functions. This verifies Theorem \ref{thm:GreenSn} in the case of ordinary GJMS operators. These two sections are included here for the convenience of graduate students.

In Section \ref{sec:Green}, we give a unified proof of Theorem \ref{thm:GreenSn}, including both the local and non-local GJMS operators, using Gegenbauer polynomials. The basics on Gegenbauer polynomials are also briefly reviewed.

In Section \ref{sec:MFE}, we use the method of Section \ref{sec:Green} to study the action of $P_n$ on axially symmetric functions for $n$ even and prove Proposition \ref{prop:full_GHX} and Theorem \ref{thm:meanfield}. 

Sections \ref{sec:warmup}-\ref{sec:MFE} constitute the first part of the paper.  The second part  consists of Sections \ref{sec:rigidity} and \ref{sec:strong rigidity} and an appendix, which are devoted to the rigidity problem. In Section \ref{sec:rigidity}, we prove the Green function rigidity Theorem \ref{thm:rigidity} by showing that the hypersurface is umbilical. This is accomplished by computing limits of the  Green function equations along curves on the hypersurface.
Then in Section \ref{sec:strong rigidity} we study the strong rigidity problem in low dimensions and prove Theorem \ref{thm:rigidity_PMT} via the Positive Mass Theorem when $n=3,4,5$. Finally, we discuss the surface case, and prove the strong rigidity for rotationally symmetric analytic surfaces (Theorem \ref{thm:rigidity_analytic}). In the appendix, we provide another proof of Theorem \ref{thm:rigidity} (2) when $n\geq 5$, by using both the asymptotic expansion formula of Green functions due to Parker-Rosenberg \cite[Theorem 2.2]{ParkerRosenberg} and the comparison of extrinsic and intrinsic distance functions, which may be of independent interest.\\

\noindent{\bf Acknowledgments:} X. Chen is partially supported by NSFC (No.12271244). Y. Shi is partially supported by NSFC (No.12371058). The authors thank Professor Jeffery S. Case for bringing \cite{QingRaske} into their sight.

\section{Two warm-up examples}\label{sec:warmup}

We start with two simplest examples of GJMS operators on spheres, namely the Laplacian on $\Sp^2$ and the conformal Laplacian on $\Sp^n$ for $n\geq 3$.

We first derive the Green function of $-\Delta_{\Sp^2}$. We identify $\Sp^2$ with $\hat{\CC}:=\CC\cup\{\infty\}$ via the stereographic projection. For $z=x+\sqrt{-1}y$, let $\Lap_0=\frac{\partial^2}{\partial x^2}+\frac{\partial^2}{\partial y^2}=4 \frac{\partial^2}{\partial z \partial \bar z}$ be the Euclidean Laplacian and $d V_0=dx \wedge d y=\frac{\sqrt{-1}}{2}d z \wedge d \bar z$.    Then the round metric is 
$$g_{\Sp^2}=\frac{4|dz|^2}{(1+|z|^2)^2},$$
and $-\Delta_{\Sp^2}=-\frac{(1+|z|^2)^2}{4}\Lap_0$, and hence $-\Lap_{\Sp^2} u\ dV_{\Sp^2}=-\Lap_0u\ dV_0$.

We take a cut-off function $\chi$, which identically equals 1 in a neighborhood of 0 and vanishes outside a large compact set. Then for any $u \in C^\infty(\Sp^2)$ and any $w\in \CC$, we have
\begin{align*}
 \int_{\hat\CC} \Lap_{\Sp^2,z} u(z)\log|z-w|\  dV_{\Sp^2}(z)
&=\int_{\hat\CC\setminus{0}} \Lap_{\Sp^2,z} \Big((1-\chi(z))u(z)\Big)\log|z-w|\  dV_{\Sp^2}(z)\\
&\quad+\int_{\CC} \Lap_{\Sp^2,z} \Big(\chi(z)u(z)\Big)\log|z-w|\  dV_{\Sp^2}(z)\\
&=\int_{\CC} \Lap_{0,z} \Big((1-\chi(\frac{1}{z}))u(\frac{1}{z})\Big)\log|\frac{1}{z}-w|\  dV_0(z)\\
&\quad+\int_{\CC} \Lap_{0,z} \Big(\chi(z)u(z)\Big)\log|z-w|\  dV_0(z)\\
&=\int_{\CC} \Lap_{0,z} \Big((1-\chi(\frac{1}{z}))u(\frac{1}{z})\Big)\Big(\log|1-wz|-\log|z|\Big)  dV_0(z)\\
&\quad +2\pi\chi(w)u(w)\\
&=2\pi (1-\chi(w))u(w)-2\pi u(\infty)+2\pi\chi(w)u(w)\\
&= 2\pi\Big(u(w)-u(\infty)\Big).
\end{align*}
This means $\Lap_{\Sp^2,z}\Big(\frac{1}{2\pi}\log|z-w|\Big)=\delta_w-\delta_\infty$, which in turn implies that for any $u\in C^\infty(\Sp^2)$,
\begin{align*}
\int_{\hat\CC}\Lap_{\Sp^2} u(z)\frac{1}{2\pi}\log\Big|\frac{1}{z}-\frac{1}{w}\Big| dV_{\Sp^2}(z)&= u(w)-u(0).
\end{align*}
Then integrate the above identity with respect to $w$ and interchange two integrals:
\begin{align*}
\int_{\hat\CC}u\ dV_{\Sp^2}-|\Sp^2|u(0)=&\int_{\hat\CC}\Lap_{\Sp^2} u(z)\Big(\int_{\hat\CC}\frac{1}{2\pi}\log\Big|\frac{1}{z}-\frac{1}{w}\Big| dV_{\Sp^2}(w)\Big)dV_{\Sp^2}(z).
\end{align*}
This means
$$\Lap_{\Sp^2,z}\Big(\frac{1}{|\Sp^2|}\int_{\hat\CC}\frac{1}{2\pi}\log\Big|\frac{1}{z}-\frac{1}{w}\Big| dV_{\Sp^2}(w)\Big)=-\delta_0(z)+\frac{1}{|\Sp^2|}.$$
So the Green function for $-\Delta_{\Sp^2}$ is
\begin{align*}
G(z)=&\frac{1}{|\Sp^2|}\int_{\hat\CC}\frac{1}{2\pi}\log\Big|\frac{1}{z}-\frac{1}{w}\Big| dV_{\Sp^2}(w)+C\\
=&-\frac{1}{2\pi}\log|z|+\frac{1}{8\pi^2}\int_{\hat\CC}\log|z-w|dV_{\Sp^2}(w)+C.
\end{align*}
Direct computation using polar coordinates gives us:  
$$\int_{\hat\CC}\log|z-w|dV_{\Sp^2}(w)=2\pi\log(1+|z|^2).$$
So we obtain
$$G(z)=-\frac{1}{2\pi}\log|z|+\frac{1}{4\pi}\log(1+|z|^2)+C.$$
Notice that 
$$G(z)=G(z,0)=-\frac{1}{2\pi}\log\frac{|z|}{\sqrt{1+|z|^2}}+C$$
is intimately connected with the extrinsic geometry of $\Sp^2$ inside $\RR^3$. In fact if we regard $Q$ as the south pole of $\Sp^2$, corresponding to $0\in\CC$, and let $P=(x_1,x_2,x_3)$ be the point corresponding to $z\in\CC$, then the extrinsic distance between $P$ and $Q$ is exactly
$$\|P-Q\|=\sqrt{x_1^2+x_2^2+(1+x_3)^2}=\sqrt{2(1+x_3)}=\frac{2|z|}{\sqrt{1+|z|^2}}.$$
So we can write
\begin{equation}
G(P,Q)=-\frac{1}{2\pi}\log\|P-Q\|+C.    
\end{equation}

\medskip

For $\Sn$, we also use the stereographic projection from the north pole to identify it with $\RR^n\cup\{\infty\}$. The round metric is
$$g_{\Sn}=\frac{4|dx|^2}{(1+|x|^2)^2}.$$
This means that if we let $\phi(x):=\Big(\frac{1+|x|^2}{2}\Big)^{\frac{n-2}{2}}$, then  $P_2\phi=0$, since $\phi^{\frac{4}{n-2}}g_{\Sn}=|dx|^2$ is just the Euclidean metric. However, this function blows up at the north pole of $\Sn$, we need to check whether there is a Dirac measure at the north pole when applying $P_2$ to $\phi$.
For this, we work with coordinates around the north pole: $y:=\frac{x}{|x|^2}$. Then under new coordinates, we have
$$g_{\Sn}=\frac{4 |dy|^2}{(1+|y|^2)^2}\qquad \mathrm{~~and~~}\qquad \phi(y)=\Big(\frac{1+|y|^2}{2|y|^2}\Big)^{\frac{n-2}{2}}.$$
 We switch back to the $x$-coordinates, and recall that in $\RR^n$ ($n\geq 3$), we have
$$-\Lap_{0,x}\frac{c_{n,1}}{|x-y|^{n-2}}=\delta_y(x) \qquad\mathrm{~~where~~} \quad c_{n,1}=\frac{1}{(n-2)|\Sp^{n-1}|}=\frac{\Gamma(\frac{n}{2}-1)}{2^2\pi^{\frac{n}{2}}}.$$
  This gives us a hint that the Green function we want has something to do with 
\begin{equation}\label{eqn:phi for Sn}
\varphi(x):=c_{n,1}\Big(\frac{1+|x|^2}{2|x|^2}\Big)^{\frac{n-2}{2}}.    
\end{equation}

\begin{proposition}
Let $\phi(x)$ be defined as \eqref{eqn:phi for Sn}, then we have
$$P_2 \phi=2^{\frac{n-2}{2}}\delta_0.$$
\end{proposition}

\begin{proof}
First note that $\phi$ is smooth in a neighborhood of $\infty$, and outside $0$, we have
\begin{align*}
\Lap_{\Sn} \phi(x)=\big(\frac{2}{1+|x|^2}\big)^{-n}\partial_i\Big(\big(\frac{2}{1+|x|^2}\big)^{n-2}\partial_i \phi\Big)=\frac{n(n-2)}{4}\phi,
\end{align*}
so we have
$$P_2\phi=-\frac{n(n-2)}{4}\phi+\frac{n-2}{4(n-1)}n(n-1)\phi=0.$$
Alternatively, this follows directly from the transformation rule for conformal Laplacian.
For any $u\in C^\infty(\Sn)$ vanishing in a neighborhood of the north pole $N$, by an abuse of notation, we identify it as a smooth function with compact support on $\RR^n$. Then using the conformal invariance of $P_2^g$ that
$$P_{2}u(x)=\Big(\frac{1+|x|^2}{2}\Big)^{\frac{n+2}{2}}(-\Delta_0)\Big(u(x)\big(\frac{2}{1+|x|^2}\big)^{\frac{n-2}{2}}\Big),$$
we have
\begin{align*}
\int_{\Sn} \phi P_2 u dV_{\Sp^n} 
=&\int_{\RR^n} \phi(x)\Big(\frac{2}{1+|x|^2}\Big)^{\frac{n-2}{2}} (-\Delta_0)\Big(u(x)\big(\frac{2}{1+|x|^2}\big)^{\frac{n-2}{2}}\Big) dx\\
=&\int_{\RR^n} \frac{c_{n,1}}{|x|^{n-2}}(-\Delta_0)\Big(u(x)\big(\frac{2}{1+|x|^2}\big)^{\frac{n-2}{2}}\Big) dx\\
=&2^{\frac{n-2}{2}}u(0),
\end{align*}
where we have used $-\Delta_0 (c_{n,1}|x|^{2-n})=\delta_0$ in the last identity.
\end{proof}

\begin{corollary}
The Green function for the conformal Laplacian on $\Sn$ ($n\geq 3$) is 
\begin{equation}
G(x)=\frac{2^{\frac{2-n}{2}}}{(n-2)|\Sp^{n-1}|}\Big(\frac{1+|x|^2}{2|x|^2}\Big)^{\frac{n-2}{2}}=c_{n,1}\Big(\frac{\sqrt{1+|x|^2}}{2|x|}\Big)^{n-2}.
\end{equation}
\end{corollary}

\begin{remark}
For the same reason as in the $\Sp^2$ case, given any  $P,Q\in \Sn$, we have
$$G(P,Q)=\frac{c_{n,1}}{\|P-Q\|^{n-2}},$$
where $\|P-Q\|$ is  understood as the extrinsic distance between $P,Q$ in $\RR^{n+1}$.

\end{remark}

The preceding discussion motivates the following plausible conjecture: The Green function of GJMS operator $P_n$ in critical dimension $n$ should be
$$G(P,Q)=-c_n \log \|P-Q\|+C $$
and for $2k<n$, the Green function of $P_{2k}$ should be
$$G(P,Q)=c_{n,k}\|P-Q\|^{2k-n},$$
where $c_n,c_{n,k}$ are the coefficients of the corresponding fundamental solutions of $(-\Delta_0)^{\frac{n}{2}}$ and $(-\Delta_0)^k$ on $\mathbb{R}^n$. We shall verify this in the next two sections.

\section{Green functions for ordinary GJMS operators on $\Sn$}\label{sec:ordinary}

Now we begin the proof of Theorem \ref{thm:GreenSn}. In this section, we study the Green functions of the ordinary GJMS operators constructed in \cite{GJMS}. In the next section, we shall treat the fractional GJMS operators.

\subsection{GJMS operators on critical dimensional spheres}

Assume $n$ is even. As before, without loss of generality, we assume that $Q$ is the south pole, so that under stereographic projection $\pi:\Sn\setminus\{N\}\to\RR^n$, we have $\pi(Q)=0$.
Let $c_n$ be the constant such that $-c_n\log|z|$ is the fundamental solution of $(-\Delta_0)^{n/2}$. Explicitly, we have
\begin{equation}\label{eqn:c_n}
  c_n=\frac{2}{\Gamma(n)|\Sn|}=\frac{1}{2^{n-1}\pi^{\frac{n}{2}}\Gamma(\frac{n}{2})}.  
\end{equation}

\begin{proposition}\label{prop:GreenSn}
    Let $G(P,Q)=-c_n\log\|P-Q\|$ as above, we have 
     \begin{equation}\label{eqn:GreenSn}
     u(Q)=\frac{1}{|\Sn|}\int_{\Sn}u dV_{\Sn}+\int_{\Sn} G(\cdot, Q) P_nu dV_{\Sn},   
    \end{equation}
    for all $u\in C^\infty(\Sn)$.
\end{proposition}

\begin{proof}[Proof of Proposition \ref{prop:GreenSn}]
    We shall use the following identity, which is equivalent to the fact that the $Q$-curvature  $Q_{\Sn}=(n-1)!$ (see, for example, \cite[Theorem 1.2]{Chang-Yang} and \cite[Theorem 1.3]{Xu}):
    $$(-\Delta_0)^{n/2}\log(1+|z|^2)=-(n-1)!\Big(\frac{2}{1+|z|^2}\Big)^n,\quad z \in \RR^n.$$
    Alternatively, since both sides are real analytic functions of $|z|$, one could verify this via Taylor expansion.

    Writing $u$ as $\chi u+(1-\chi)u$ we can assume  either $u$ vanishes in a neighborhood of $Q$, or $u$ vanishes in a neighborhood of the antipodal point of $Q$, i.e. the north pole $N$.
    
    \medskip
    \noindent\underline{Case 1:} $u$ vanishes in a neighborhood of $Q$.

    In this case, we use the stereographic projection form $Q$, $\tilde\pi:\Sn\setminus\{Q\}\to\RR^n$. Then $u\circ\tilde\pi^{-1}$ is a smooth function on $\RR^n$ with compact support. Then we have
    \begin{align*}
        \int_{\Sn} G(\cdot, Q) P_nu dV_{\Sn}&=\int_{\RR^n} G(\cdot,Q)\circ \tilde\pi^{-1}\ (-\Delta_0)^{\frac{n}{2}} (u\circ\tilde\pi^{-1}) dz\\
        &=\int_{\RR^n} G(\cdot,Q)\circ \tilde\pi^{-1}\ (-\Delta_0)^{\frac{n}{2}} (u\circ\tilde\pi^{-1}) dz\\
        &=\int_{\RR^n} -c_n \log\sqrt{(1+x_{n+1})^2+x_1^2+\dots+x_n^2}(-\Delta_0)^{\frac{n}{2}} (u\circ\tilde\pi^{-1}) dz\\
        &=\int_{\RR^n} -c_n\log\frac{2}{\sqrt{1+|z|^2}}(-\Delta_0)^{\frac{n}{2}} (u\circ\tilde\pi^{-1})(z) dz\\
        &=\frac{c_n}{2}\int_{\RR^n}  (-\Delta_0)^{n/2}\log(1+|z|^2) u\circ\tilde\pi^{-1}(z) dz\\
        &=\frac{1}{2^{n}\Gamma(\frac{n}{2})\pi^{\frac{n}{2}}}\int_{\RR^n} \frac{-2^n(n-1)!}{(1+|z|^2)^n} u\circ\tilde\pi^{-1}(z) dz\\
        &=-\frac{(n-1)!}{2^{n}\Gamma(\frac{n}{2})\pi^{\frac{n}{2}}}\int_{\Sn}u dV_{\Sn}.
    \end{align*}
    On the other hand,
    $$|\Sn|=(n+1)\frac{\pi^{\frac{n+1}{2}}}{\Gamma(\frac{n+1}{2}+1)}=\frac{2^n\pi^{\frac{n}{2}}\Gamma(\frac{n}{2})}{(n-1)!}.$$
    So we obtain
    $$\int_{\Sn} G(\cdot, Q) P_nu dV_{\Sn}+\frac{1}{|\Sn|}\int_{\Sn}u dV_{\Sn}=0=u(Q).$$

\medskip

    \noindent\underline{Case 2:} $u$ vanishes in a neighborhood of the antipodal point of $Q$.

    In this case, $u\circ\pi^{-1}$ is a smooth function with compact support on $\RR^n$. Similarly, we have
    \begin{align*}
        \int_{\Sn} G(\cdot, Q) P_nu dV_{\Sn}&=\int_{\RR^n} G(\cdot,Q)\circ \pi^{-1}\ P_{n,\RR^n} (u\circ\pi^{-1}) dz\\
        &=\int_{\RR^n} G(\cdot,Q)\circ \pi^{-1}\ (-\Delta_0)^{\frac{n}{2}} (u\circ\pi^{-1}) dz\\
        &=\int_{\RR^n} -c_n \log\sqrt{(1+x_{n+1})^2+x_1^2+\dots+x_n^2}(-\Delta_0)^{\frac{n}{2}} (u\circ\pi^{-1}) dz\\
        &=\int_{\RR^n} -c_n\log\frac{2|z|}{\sqrt{1+|z|^2}}(-\Delta_0)^{\frac{n}{2}} (u\circ\pi^{-1})(z) dz\\
        &=u\circ\pi^{-1}(0)+\frac{c_n}{2}\int_{\RR^n}  (-\Delta_0)^{n/2}\log(1+|z|^2) u\circ\pi^{-1}(z) dz\\
        &=u(Q)+\frac{1}{2^{n}\Gamma(\frac{n}{2})\pi^{\frac{n}{2}}}\int_{\RR^n} \frac{-2^n(n-1)!}{(1+|z|^2)^n} u\circ\pi^{-1}(z) dz\\
        &=u(Q)-\frac{1}{|\Sn|}\int_{\Sn}u dV_{\Sn}.
    \end{align*}
    This concludes the proof of Theorem \ref{thm:GreenSn}.
\end{proof}

\subsection{Green function for GJMS operator on non-critical dimensional spheres}
Recall that for $2k<n$ or $2k>n$ when $n$ is odd, the fundamental solution of $(-\Delta_0)^k$ on $\RR^n$ is 
$$\Gamma(x,y)=c_{n,k}|x-y|^{2k-n},$$
where
$$c_{n,k}=\frac{\Gamma(\frac{n}{2}-k)}{2^{2k}\pi^{\frac{n}{2}}\Gamma(k)}.$$

\begin{proposition}\label{prop:subGreenSn}
    Let $2k<n$ or $2k>n$ when $n$ is odd, and for $P,Q\in \Sn\subset\RR^{n+1}$, set $G(P,Q)=c_{n,k}\|P-Q\|^{2k-n}$. Then we have 
     \begin{equation*}
     u(Q)=\int_{\Sn} G(\cdot, Q) P_{2k}u dV_{\Sn},  
    \end{equation*} 
    for all $u\in C^\infty(\Sn)$.
\end{proposition}

Though some special cases have been known before, we provide a unified simple approach here for readers' convenience.

\begin{proof}
    We follow the same route as the proof of Proposition \ref{prop:GreenSn}, namely we can assume that $Q$ is the south pole and only need to consider two special cases when $u$ vanishes either near $Q$ or near the antipodal point of $Q$.

    \medskip

    \noindent\underline{Case 1:} $u$ vanishes in a neighborhood of $Q$.

    As before, we use the stereographic projection from $Q$, $\tilde\pi:\Sn\setminus\{Q\}\to\RR^n$. Then $u\circ\tilde\pi^{-1}$ is a smooth function on $\RR^n$ with compact support. Then we have
    \begin{align*}
        \int_{\Sn} G(\cdot, Q) P_{2k}u dV_{\Sn}&=\int_{\RR^n} \big(\frac{2}{1+|z|^2}\big)^{\frac{n}{2}-k}G(\cdot,Q)\circ \tilde\pi^{-1}\ P_{2k ,\RR^n} \Big(\big(\frac{2}{1+|z|^2}\big)^{\frac{n}{2}-k}u\circ\tilde\pi^{-1}\Big) dz\\
        &=\int_{\RR^n}  \big(\frac{2}{1+|z|^2}\big)^{\frac{n}{2}-k}G(\cdot,Q)\circ \tilde\pi^{-1} (-\Delta_0)^{k} \Big(\big(\frac{2}{1+|z|^2}\big)^{\frac{n}{2}-k}u\circ\tilde\pi^{-1}\Big) dz\\
        &=\int_{\RR^n} c_{n,k} \big(\frac{4}{1+|z|^2}\big)^{k-\frac{n}{2}}\big(\frac{2}{1+|z|^2}\big)^{\frac{n}{2}-k} (-\Delta_0)^{k} \Big(\big(\frac{2}{1+|z|^2}\big)^{\frac{n}{2}-k}u\circ\tilde\pi^{-1}\Big) dz\\
        &=c_{n,k}2^{k-\frac{n}{2}}\int_{\RR^n}(-\Delta_0)^{k} \Big(\big(\frac{2}{1+|z|^2}\big)^{\frac{n}{2}-k}u\circ\tilde\pi^{-1}\Big) dz\\
        &=0.
    \end{align*}
    
    So we obtain
    $$\int_{\Sn} G(\cdot, Q) P_{2k}u dV_{\Sn}=0=u(Q).$$
    
    \medskip

    \noindent\underline{Case 2:} $u$ vanishes in a neighborhood of the antipodal point of $Q$.

    In this case, $u\circ\pi^{-1}$ is a smooth function with compact support on $\RR^n$, similarly we have
    \begin{align*}
        \int_{\Sn} G(\cdot, Q) P_{2k}u dV_{\Sn}&=\int_{\RR^n} \big(\frac{2}{1+|z|^2}\big)^{\frac{n}{2}-k}G(\cdot,Q)\circ \pi^{-1}\ P_{2k ,\RR^n} \Big(\big(\frac{2}{1+|z|^2}\big)^{\frac{n}{2}-k}u\circ\pi^{-1}\Big) dz\\
        &=\int_{\RR^n}  \big(\frac{2}{1+|z|^2}\big)^{\frac{n}{2}-k}G(\cdot,Q)\circ \pi^{-1} (-\Delta_0)^{k} \Big(\big(\frac{2}{1+|z|^2}\big)^{\frac{n}{2}-k}u\circ\pi^{-1}\Big) dz\\
        &=\int_{\RR^n} c_{n,k} \big(\frac{4|z|^2}{1+|z|^2}\big)^{k-\frac{n}{2}}\big(\frac{2}{1+|z|^2}\big)^{\frac{n}{2}-k} (-\Delta_0)^{k} \Big(\big(\frac{2}{1+|z|^2}\big)^{\frac{n}{2}-k}u\circ\pi^{-1}\Big) dz\\
        &=2^{k-\frac{n}{2}}\int_{\RR^n}c_{n,k}|z|^{2k-n}(-\Delta_0)^{k} \Big(\big(\frac{2}{1+|z|^2}\big)^{\frac{n}{2}-k}u\circ\pi^{-1}\Big) dz\\
        &=2^{k-\frac{n}{2}}2^{\frac{n}{2}-k}u\circ\pi^{-1}(0)=u(Q).
    \end{align*}
    This concludes the proof of Proposition \ref{prop:subGreenSn}.
\end{proof}

\section{Gegenbauer polynomials 
 and Green functions of non-local GJMS operators on spheres}\label{sec:Green}

\subsection{Preliminaries on Gegenbauer polynomials}\label{subsec:Gegenbauer}

We now recall some basic properties of  spherical harmonics and Gegenbauer polynomials that we shall use later. Our main references are \cite[Chapter 4]{Stein-Weiss} and  \cite[Chapter 2]{Hua}, see also \cite[Chapter 7]{Hua_Monog}. Remember that since we work on $\Sn$ instead of $\Sp^{n-1}$, the notation here is a little bit different. 

For $\lambda>-\frac{1}{2}$, the Gegenbauer polynomial $P^\lambda_k$ is a polynomial of degree $k$, and can be defined by a generating function
$$\frac{1}{(1-2tx+t^2)^\lambda}=\sum_{k=0}^\infty P^\lambda_k(x)t^k.$$
More explicitly, 
\begin{align*}
P^\lambda_k(x)=\frac{(-2)^k}{k!}\frac{\Gamma(k+\lambda)\Gamma(k+2\lambda)}{\Gamma(\lambda)\Gamma(2k+2\lambda)}(1-x^2)^{\frac{1}{2}-\lambda}\frac{d^k}{dx^k} (1-x^2)^{k-\frac{1}{2}+\lambda}
\end{align*}
and
\begin{align*}
\int_{-1}^1 P^\lambda_k(x) P^\lambda_l (x)(1-x^2)^{\lambda-\frac{1}{2}}dx=\frac{2^{1-2\lambda}\pi \Gamma(k+2\lambda)}{\Gamma(\lambda)^2(k+\lambda)\Gamma(k+1)} \delta_{kl}.
\end{align*}
See Hua \cite[formulae $(8)$ on p.34 and $(4)$ on p.37]{Hua} for two formulae above.

The relation between Gegenbauer polynomials and the analysis on spheres is as follows: Let $\mathcal{H}_k$ be the space of  spherical harmonics of degree $k$ on $\Sn$. Then 
$$\dim\mathcal{H}_k:=N_k:=\binom{n+k}{k}-\binom{n+k-2}{k-2}=\frac{(n+2k-1)\Gamma(n+k-1)}{\Gamma(n)\Gamma(k+1)}.$$ 
Choose an $L^2(\Sn)$-orthonormal basis $Y^{(k)}_m$, $m=1,\cdots, N_k$ of $\mathcal{H}_k$, for any $P,Q\in \Sn$, it follows from  \cite[Theorem 2.14 of Chapter 4]{Stein-Weiss} that
\begin{equation}\label{eqn:FunkHecke}
\sum_{m=1}^{N_k}\overline{Y^{(k)}_m(Q)}Y^{(k)}_m(P)=c_{k,n}P^\lambda_k(P\cdot Q),    
\end{equation}
where $P^\lambda_k$ is the Gegenbauer polynomial of degree $k$ with $\lambda=\frac{n-1}{2}$,
and $c_{k,n}$ is a constant. See also \cite[equation (7.3.9)]{Hua_Monog}. This is a special case of the so-called `\emph{Funk-Hecke theorem}'. Furthermore, the constant $c_{k,n}$ can be explicitly calculated:  Choose $P=Q$ and integrate both sides of \eqref{eqn:FunkHecke} to show
\begin{align*}
  c_{k,n}&=\frac{N_k}{P^\lambda_k(1)|\Sn|}
  =\frac{1}{|\Sn|}\frac{(n+2k-1)\Gamma(n+k-1)}{\Gamma(n)\Gamma(k+1)}\frac{\Gamma(k+1)\Gamma(n-1)}{\Gamma(n+k-1)}=\frac{n+2k-1}{(n-1)|\Sn|}.
\end{align*}
Note that this kind of phenomenon also occurs in the study of Bergman kernels of domains or K\"ahler metrics with $U(n)$ symmetries. See for example \cite[\S 1.4]{Krantz} and \cite{Lu-Tian}.

\subsection{Green functions for $P_n$ on $\Sn$}

We first study the Green function for $P_n$ on $\Sn$ for all $n \geq 2$.

\begin{proof}[Proof of Theorem \ref{thm:GreenSn} (1)]
 Recall that if $Y$ is a spherical harmonic function of degree $k$ on $\Sn$, then $-\Delta_{\Sn} Y=k(k+n-1)Y_k$ and consequently $P_n Y=\lambda_k Y$, where $\lambda_k=\Gamma(n+k)/\Gamma(k)$.

Based on this, by spectral theory for elliptic operators, we conclude that, for $Q\in \Sn$ the south pole and $P\in \Sn$ an arbitrary point, the Green function is
$$G(P,Q)=\sum_{k=1}^\infty\frac{1}{\lambda_k}\sum_{m=1}^{N_k}\overline{Y^{(k)}_m(Q)}Y^{(k)}_m(P).$$
Now, by \eqref{eqn:FunkHecke}, we have
\begin{align*}
G(P,Q)&=\sum_{k=1}^\infty \frac{c_{k,n}}{\lambda_k}P^\lambda_k(P\cdot Q)=\sum_{k=1}^\infty \frac{c_{k,n}}{\lambda_k}P^\lambda_k(-x)\\
&=\frac{1}{(n-1)|\Sn|}\sum_{k=1}^\infty\frac{(-1)^k(n+2k-1)\Gamma(k)}{\Gamma(n+k)}P^\lambda_k(x)=:f(x),
\end{align*}
where we denote $x:-P\cdot Q$ in the second equality.

\medskip
\noindent{\bf Claim:} There is a constant $c$ such that $$G(P,Q)=f(x)=-\frac{c_n}{2}\log(1+x)+c,$$
where $c_n$ is the same as Theorem \ref{thm:GreenSn} (1).
\medskip

Since both $\log(1+x)$ and $f(x)$ are elements of $L^2([-1,1], d\mu)$, it suffices to check that both sides have the same inner product with each $P^\lambda_k, k\geq 1$.

On one hand, we have
\begin{align*}
    \int_{-1}^1 f(x)P^\lambda_k(x)d\mu &= \frac{1}{(n-1)|\Sn|}\frac{(-1)^k(n+2k-1)\Gamma(k)}{\Gamma(n+k)}\|P^\lambda_k\|_{L^2([-1,1],d\mu)}^2\\
    &= \frac{(-1)^k}{(n-1)|\Sn|}\frac{(n+2k-1)\Gamma(k)}{\Gamma(n+k)}\frac{2^{2-n}\pi \Gamma(k+n-1)}{\Gamma(\frac{n-1}{2})^2(k+\frac{n-1}{2})\Gamma(k+1)}\\
    &= \frac{\Gamma(\frac{n-1}{2})}{4\pi^{\frac{n+1}{2}}}\frac{(-1)^k2^{3-n}\pi}{\Gamma(\frac{n-1}{2})^2(n+k-1)k}\\
    &=\frac{1}{(4\pi)^{\frac{n-1}{2}}\Gamma(\frac{n-1}{2})}\frac{(-1)^k}{k(n+k-1)}.
\end{align*}
On the other hand, we have
\begin{align*}
  \int_{-1}^1 \log(1+x)P^\lambda_k(x)d\mu &=\frac{(-2)^k}{k!}\frac{\Gamma(k+\frac{n-1}{2})\Gamma(k+n-1)}{\Gamma(\frac{n-1}{2})\Gamma(2k+n-1)}\cdot\\
  &\quad \int_{-1}^1\log(1+x)\frac{d^k}{dx^k} (1-x^2)^{k+\frac{n-2}{2}} dx,
\end{align*}
where the integration by parts gives
\begin{align*}
  \int_{-1}^1\log(1+x)\frac{d^k}{dx^k} (1-x^2)^{k+\frac{n-2}{2}} dx&=-\int_{-1}^1\frac{1}{1+x}  \frac{d^{k-1}}{dx^{k-1}} (1-x^2)^{k+\frac{n-2}{2}} dx\\
  &=-\int_{-1}^1\frac{1}{(1+x)^2}  \frac{d^{k-2}}{dx^{k-2}} (1-x^2)^{k+\frac{n-2}{2}} dx\\
  &=\cdots\\
  &=-\int_{-1}^1(k-1)!(1+x)^{-k} (1-x^2)^{k+\frac{n-2}{2}} dx\\
  &=-(k-1)!\int_{-1}^1(1+x)^{\frac{n-2}{2}}(1-x)^{k+\frac{n-2}{2}} dx\\
  &=-\Gamma(k)2^{k+n-1}\frac{\Gamma(\frac{n}{2})\Gamma(k+\frac{n}{2})}{\Gamma(k+n)}.
\end{align*}
Using the fact that $\Gamma(x)\Gamma(x+\frac{1}{2})=2^{1-2x}\sqrt{\pi}\Gamma(2x)$, we get
$$\int_{-1}^1\log(1+x)\frac{d^k}{dx^k} (1-x^2)^{k+\frac{n-2}{2}} dx=-\frac{2\sqrt{\pi}\Gamma(\frac{n}{2})}{\Gamma(\frac{n-1}{2})}\frac{(-1)^k}{k(k+n-1)}.$$
Hence we obtain
$$f(x)=-\frac{1}{(4\pi)^{\frac{n}{2}}\Gamma(\frac{n}{2})}\log(1+x)+c,$$
from which we conclude that
$$G(P,Q)=-\frac{1}{(4\pi)^{\frac{n}{2}}\Gamma(\frac{n}{2})}\log\|P-Q\|^2+C=-\frac{2}{(n-1)!|\Sn|}\log\|P-Q\|^2+c$$
for some constant $c$.   
\end{proof}

\subsection{Green function for $P_{2\sigma}$ on non-critical dimensional $\Sn$}

Since case (2) of Theorem \ref{thm:GreenSn} is contained in case (3), it reduces to the proof of case (3).

\begin{proof}[Proof of Theorem \ref{thm:GreenSn} (3)]
For $\sigma\in (0,\frac{n}{2})\cup (\frac{n}{2},+\infty)$ and $\sigma \notin \frac{n}{2}+ \N$, we consider the fractional GJMS operator $P_{2\sigma}$ on $\Sn$ and define its Green function in the distribution sense as in \eqref{Def:Green_fcn_others}. 

For a spherical harmonic $Y_k$ of degree $k$ on $\Sn$, it is known (see, for example \cite[p.3691]{Branson95} or \cite[p.479]{Morpurgo}) that 
$$P_{2\sigma}Y_k=\frac{\Gamma(k+\frac{n}{2}+\sigma)}{\Gamma(k+\frac{n}{2}-\sigma)}Y_k:=\lambda_k Y_k.$$
Since $\sigma\notin\frac{n}{2}+\mathbb{N}$, $\lambda_k$ is never 0 and hence $\mathrm{Ker}(P_{2\sigma})=\{0\}$.

As in the preceding section, for $Q\in \Sn$ the south pole and $P\in \Sn$ an arbitrary point, let $x=-P \cdot Q$ and  $\lambda=\frac{n-1}{2}$, the Green function of $P_{2\sigma}$ can be written as 
\begin{align*}
f(x):=&\sum_{k=0}^\infty\frac{1}{\lambda_k}\sum_{m=1}^{N_k}\overline{Y^{(k)}_m(Q)}Y^{(k)}_m(P)\\
=&\frac{1}{(n-1)|\Sn|}\sum_{k=0}^\infty \frac{(-1)^k(n+2k-1)\Gamma(k+\frac{n}{2}-\sigma)}{\Gamma(k+\frac{n}{2}+\sigma)}P^\lambda_k(x).
\end{align*}

\noindent{\bf Claim:} We have $$G(P,Q)=f(x)=c_{n,\sigma}(1+x)^{\sigma-\frac{n}{2}},$$
where 
$$c_{n,\sigma}=\frac{\Gamma(\frac{n}{2}-\sigma)}{2^{2\sigma}\pi^{\frac{n}{2}}\Gamma(\sigma)}.$$

To this end, as before we have
\begin{align*}
 \int_{-1}^1 f(x)P^\lambda_k(x)d\mu &= \frac{1}{(n-1)|\Sn|}\frac{(-1)^k(n+2k-1)\Gamma(k+\frac{n}{2}-\sigma)}{\Gamma(k+\frac{n}{2}+\sigma)}\|P^\lambda_k\|_{L^2([-1,1],d\mu)}^2\\
    &= \frac{1}{(n-1)|\Sn|}\frac{(-1)^k(n+2k-1)\Gamma(k+\frac{n}{2}-\sigma)}{\Gamma(k+\frac{n}{2}+\sigma)}\\
    &\qquad\cdot\frac{2^{2-n}\pi \Gamma(k+n-1)}{\Gamma(\frac{n-1}{2})^2(k+\frac{n-1}{2})\Gamma(k+1)}\\
    &= \frac{2^{3-n}\pi}{(n-1)|\Sn|\Gamma(\frac{n-1}{2})^2}\frac{(-1)^k\Gamma(k+n-1)\Gamma(k+\frac{n}{2}-\sigma)}{\Gamma(k+1)\Gamma(k+\frac{n}{2}+\sigma)}.
\end{align*}

On the other hand, we have
\begin{align*}
  \int_{-1}^1 (1+x)^{\sigma-\frac{n}{2}}P^\lambda_k(x)d\mu &=\frac{(-2)^k}{k!}\frac{\Gamma(k+\frac{n-1}{2})\Gamma(k+n-1)}{\Gamma(\frac{n-1}{2})\Gamma(2k+n-1)}\cdot\\
  &\quad \int_{-1}^1(1+x)^{\sigma-\frac{n}{2}}\frac{d^k}{dx^k} (1-x^2)^{k+\frac{n-2}{2}} dx,
\end{align*}
where
\begin{align*}
  \int_{-1}^1(1+x)^{\sigma-\frac{n}{2}}\frac{d^k}{dx^k} (1-x^2)^{k+\frac{n-2}{2}} dx&=(\frac{n}{2}-\sigma)(\frac{n}{2}+1-\sigma)\cdots(\frac{n}{2}+k-1-\sigma)\\
  &\qquad\cdot \int_{-1}^1 (1+x)^{\sigma-\frac{n}{2}-k}(1-x^2)^{k+\frac{n-2}{2}}dx\\
  &=\frac{\Gamma(\frac{n}{2}+k-\sigma)}{\Gamma(\frac{n}{2}-\sigma)}\int_{-1}^1 (1+x)^{\sigma-1}(1-x)^{k+\frac{n-2}{2}}dx\\
  &=\frac{\Gamma(\frac{n}{2}+k-\sigma)}{\Gamma(\frac{n}{2}-\sigma)}\cdot\frac{2^{\sigma+k+\frac{n}{2}-1}\Gamma(\sigma)\Gamma(k+\frac{n}{2})}{\Gamma(\sigma+k+\frac{n}{2})}.
\end{align*}
Again, using $$\Gamma(k+\frac{n-1}{2})\Gamma(k+\frac{n}{2})=2^{2-2k-n}\sqrt{\pi}\Gamma(2k+n-1),$$
we conclude that
\begin{align*}
    \int_{-1}^1 (1+x)^{\sigma-\frac{n}{2}}P^\lambda_k(x)d\mu
    &=\frac{(-2)^k}{k!}\frac{\Gamma(k+\frac{n-1}{2})\Gamma(k+n-1)}{\Gamma(\frac{n-1}{2})\Gamma(2k+n-1)}\cdot \frac{\Gamma(\frac{n}{2}+k-\sigma)}{\Gamma(\frac{n}{2}-\sigma)}\\
    &\qquad\cdot\frac{2^{\sigma+k+\frac{n}{2}-1}\Gamma(\sigma)\Gamma(k+\frac{n}{2})}{\Gamma(\sigma+k+\frac{n}{2})}\\
    &=\frac{2^{\sigma-\frac{n}{2}+1}\sqrt{\pi}\Gamma(\sigma)}{\Gamma(\frac{n-1}{2})\Gamma(\frac{n}{2}-\sigma)}\cdot\frac{(-1)^k\Gamma(k+n-1)\Gamma(k+\frac{n}{2}-\sigma)}{\Gamma(k+1)\Gamma(k+\frac{n}{2}+\sigma)}.
\end{align*}
In turn, this implies
\begin{align*}
G(P,Q)=f(x)&=\frac{2^{3-n}\pi}{(n-1)|\Sn|\Gamma(\frac{n-1}{2})^2}\cdot \frac{\Gamma(\frac{n-1}{2})\Gamma(\frac{n}{2}-\sigma)}{2^{\sigma-\frac{n}{2}+1}\sqrt{\pi}\Gamma(\sigma)}(1+x)^{\sigma-\frac{n}{2}} \\
&=\frac{2^{2-\sigma-\frac{n}{2}}\sqrt{\pi}\Gamma(\frac{n}{2}-\sigma)}{(n-1)|\Sn|\Gamma(\frac{n-1}{2})\Gamma(\sigma)}(1+x)^{\sigma-\frac{n}{2}}\\
&=\frac{2^{-\sigma-\frac{n}{2}}\Gamma(\frac{n}{2}-\sigma)}{\pi^{\frac{n}{2}}\Gamma(\sigma)}(1+x)^{\sigma-\frac{n}{2}}\\
&=\frac{2^{-\sigma-\frac{n}{2}}\Gamma(\frac{n}{2}-\sigma)}{\pi^{\frac{n}{2}}\Gamma(\sigma)}\Big(\frac{\|P-Q\|^{2}}{2}\Big)^{\sigma-\frac{n}{2}}\\
&=\frac{\Gamma(\frac{n}{2}-\sigma)}{2^{2\sigma}\pi^{\frac{n}{2}}\Gamma(\sigma)}\|P-Q\|^{2\sigma-n}.
\end{align*}

This finishes the proof of Theorem \ref{thm:GreenSn}. 
\end{proof}

\section{A clean formula for $P_n$ on axially symmetric functions and mean field equation}\label{sec:MFE}

As an application of the method in the previous section, we shall study a higher-order mean field type equation on $\Sp^n$ for $n$ even.

Suppose a smooth function $u$ depends only on $x:=x_{n+1}$ with $n$ even, we are now in a position to prove Proposition \ref{prop:full_GHX}, namely 
\begin{equation}\label{formula:GHX}
P_n u(x)=(-1)^{\frac{n}{2}}[(1-x^2)^{\frac{n}{2}}u']^{(n-1)}.
\end{equation} 

\begin{proof}[Proof of Proposition \ref{prop:full_GHX}:]
    To this end, up to a constant we define
$$\mathcal{Q} u(x):=\Big[(1-x^2)^\frac{n}{2}u'(x)\Big]^{(n-1)}.$$
To check that $\mathcal{Q}$ coincides the GJMS operator $(-1)^{n/2}P_n$ in critical dimension $n$, we only need to check their actions coincide on Gegenbauer polynomials. For simplicity, we drop the normalizing constants and consider
$$u_k(x):=(1-x^2)^{-\frac{n-2}{2}}\frac{d^k}{dx^k}\Big[(1-x^2)^{k+\frac{n-2}{2}}\Big].$$
It is easy to see that $u_k$ is a polynomial of degree $k$ and $u_k, u_l, k\neq l$ are orthogonal to each other with respect to the measure $$d\mu:=(1-x^2)^{\frac{n-2}{2}}dx$$
on $[-1,1]$.

To prove $\mathcal{Q}u_k$ is a constant multiple of $u_k$, direct computation seems to be too complicated. Here we adopt the dual viewpoint as in the previous section, by noticing that it suffices to check that for any $l\neq k$, we have
$\mathcal{Q}u_k$ is orthogonal to $u_l$ with respect to $d\mu$
and that $\int_{-1}^1 \mathcal{Q}u_k(x) u_k(x)d\mu$ is the correct constant multiple of $\int_{-1}^1 u^2_k(x)d\mu$. Thus  \eqref{formula:GHX} follows from the following two elementary lemmas.
\end{proof}

\begin{lemma}\label{lem:GHX_1}
    For $k\neq l$, we have
   $$\int_{-1}^1 \mathcal{Q}u_k(x) u_l(x)d\mu=0.$$ 
\end{lemma}

\begin{lemma}\label{lem:GHX_2}
    We have
    $$\int_{-1}^1 \mathcal{Q}u_k(x) u_k(x)d\mu=\bar\lambda_k\int_{-1}^1 u^2_k(x)d\mu,$$
    where $$\bar\lambda_k=(-1)^{\frac{n}{2}}\binom{k+n-1}{n} n!=(-1)^{\frac{n}{2}}\frac{(k+n-1)!}{(k-1)!}=(-1)^{\frac{n}{2}}\frac{\Gamma(k+n)}{\Gamma(k)}.$$
\end{lemma}

\begin{proof}[Proof of Lemma \ref{lem:GHX_1}:]
    We shall compute the integral directly, by observing that when the functions vanish at boundary of $[-1,1]$, we can integrate by parts without worry. In fact, for any smooth function $f$, we always have
    $$\int_{-1}^1 f(x)u_k(x) d\mu=(-1)^k\int_{-1}^1 f^{(k)}(x)(1-x^2)^{k+\frac{n-2}{2}}dx.$$
    Notice that $\mathcal{Q}u_k$ is a polynomial of degree at most $k$, for any $l>k$. Then we have
    $$\int_{-1}^1 \mathcal{Q}u_k(x) u_l(x)d\mu=(-1)^l\int_{-1}^1\Big[\mathcal{Q}u_k\Big]^{(l)}(1-x^2)^{l+\frac{n-2}{2}}dx=0.$$
    
    Now assume $l<k$, integrating by parts gives
    \begin{align*}
       &\int_{-1}^1 \mathcal{Q}u_k(x) u_l(x)d\mu=\int_{-1}^1 \Big[(1-x^2)^\frac{n}{2}u_k'(x)\Big]^{(n-1)}\Big[(1-x^2)^{l+\frac{n-2}{2}}\Big]^{(l)} dx\\
       =&(-1)^{\frac{n-2}{2}} \int_{-1}^1 \Big[(1-x^2)^\frac{n}{2}u_k'(x)\Big]^{(\frac{n}{2})}\Big[(1-x^2)^{l+\frac{n-2}{2}}\Big]^{(l+\frac{n-2}{2})} dx\\
       =&(-1)^{n-1}\int_{-1}^1 (1-x^2)^\frac{n}{2}\frac{d}{dx}\left\{(1-x^2)^{-\frac{n-2}{2}}\frac{d^k}{dx^k}\Big[(1-x^2)^{k+\frac{n-2}{2}}\Big]\right\} \Big[(1-x^2)^{l+\frac{n-2}{2}}\Big]^{(l+n-1)} dx\\
       =&(-1)^{n-1}\int_{-1}^1 \Big[(n-2)x\frac{d^k}{dx^k}(1-x^2)^{k+\frac{n-2}{2}}\\
       &\qquad \qquad \quad+(1-x^2)\frac{d^{k+1}}{dx^{k+1}}(1-x^2)^{k+\frac{n-2}{2}}\Big]\Big[(1-x^2)^{l+\frac{n-2}{2}}\Big]^{(l+n-1)} dx\\
       :=&I+II.
    \end{align*}
    
    For $I$, since $x\Big[(1-x^2)^{l+\frac{n-2}{2}}\Big]^{(l+n-1)}$ is a polynomial of degree at most $1+l-1=l<k$, we have
    \begin{align*}
        I:=& (-1)^{n-1}\int_{-1}^1 (n-2)x\frac{d^k}{dx^k}(1-x^2)^{k+\frac{n-2}{2}}\Big[(1-x^2)^{l+\frac{n-2}{2}}\Big]^{(l+n-1)}dx\\
        =&(-1)^{n-1+k}(n-2)\int_{-1}^1(1-x^2)^{k+\frac{n-2}{2}}\frac{d^k}{dx^k}\Big(x\Big[(1-x^2)^{l+\frac{n-2}{2}}\Big]^{(l+n-1)}\Big) dx\\
        =&0.
    \end{align*}
    
    For $II$, noticing that $(1-x^2)\Big[(1-x^2)^{l+\frac{n-2}{2}}\Big]^{(l+n-1)}$ is a polynomial of degree at most $1+l<1+k$, we have, similarly
    \begin{align*}
        II:=& (-1)^{n-1}\int_{-1}^1 (1-x^2)\frac{d^{k+1}}{dx^{k+1}}(1-x^2)^{k+\frac{n-2}{2}}\Big[(1-x^2)^{l+\frac{n-2}{2}}\Big]^{(l+n-1)} dx\\
        =&(-1)^{n+k}\int_{-1}^1(1-x^2)^{k+\frac{n-2}{2}}\frac{d^{k+1}}{dx^{k+1}}\Big((1-x^2)\Big[(1-x^2)^{l+\frac{n-2}{2}}\Big]^{(l+n-1)}\Big) dx\\
        =&0.
    \end{align*}
    Therefore, the desired assertion follows by collecting these facts together.
\end{proof}

\begin{proof}[Proof of Lemma \ref{lem:GHX_2}:]
    Since both $(\mathcal{Q}u_k)^{(k)}$ and $u_k^{(k)}$ are constants, integrating by parts gives
    \begin{align*}
        \int_{-1}^1 \mathcal{Q}u_k(x) u_k(x)d\mu&= (-1)^k\int_{-1}^1 (\mathcal{Q}u_k)^{(k)} (1-x^2)^{k+\frac{n-2}{2}}dx\\
        &=(-1)^k\int_{-1}^1 \Big[(1-x^2)^\frac{n}{2}u_k'\Big]^{(k+n-1)} (1-x^2)^{k+\frac{n-2}{2}}dx\\
        &=(-1)^{k+\frac{n}{2}}\int_{-1}^1 (x^n u_k')^{(k+n-1)} (1-x^2)^{k+\frac{n-2}{2}}dx\\
        &=(-1)^{k+\frac{n}{2}}\binom{n+k-1}{n}n!\int_{-1}^1 u_k^{(k)} (1-x^2)^{k+\frac{n-2}{2}}dx\\
        &=(-1)^{\frac{n}{2}}\binom{n+k-1}{n} n!\int_{-1}^1 u^2_k(x)(1-x^2)^{\frac{n-2}{2}}dx\\
        &=\bar\lambda_k\int_{-1}^1 u^2_k(x)d\mu,
    \end{align*}
where third identity follows from the fact that $(1-x^2)^\frac{n}{2}u_k'$ is a polynomial of degree at most $k+n-1$.
\end{proof}

As an application of \eqref{formula:GHX}, we shall give an alternative proof of Theorem \ref{thm:meanfield}, namely, 
any axially symmetric smooth solutions to 
$$\frac{1}{n+1} P_n u=(n-1)!(e^{nu}-1) \qquad \mathrm{~~on~~} \quad \Sn$$
is necessarily 0.

\begin{proof}[Proof of Theorem \ref{thm:meanfield}] Write $\alpha:=\frac{1}{n+1}$, and let $u$ be a axially symmetric smooth solution of \eqref{eqn:MFE}. Without loss
of generality, we can assume that $u$ is a function of $x:=x_{n+1}$, with $x\in [-1,1]$. Now the round metric can be expressed as
$$g_{\Sn}=(1-x^2)^{-1} d x^2+(1-x^2) g_{\Sp^{n-1}}.$$
Recall \eqref{formula:GHX}:
$$P_n u=(-1)^{\frac{n}{2}}[(1-x^2)^{\frac{n}{2}}u']^{(n-1)}=(-1)^{\frac{n}{2}}[(1-x^2)^{\frac{n-2}{2}}U]^{(n-1)}$$
where $U=\langle \nabla x,\nabla u\rangle=(1-x^2)^{n/2}u'$.

Let $F_2(x)=1-(n+1)x^2$, then we multiply both sides of \eqref{eqn:MFE} by $F_2$ and integrate over $\Sn$ to show
\begin{align*}
\alpha \int_{\Sn} P_n u F_2 d V_{\Sn}=&(n-1)!\int_{\Sn}(e^{nu}-1)F_2 d V_{\Sn}\\
=&-\frac{(n-1)!}{2(n+1)}\int_{\Sn}(e^{nu}-1)\Delta_{\Sn} F_2 d V_{\Sn}\\
=&\frac{n!}{2(n+1)}\int_{\Sn}e^{nu}\langle \nabla F_2,\nabla u\rangle d V_{\Sn}\\
=&-n! \int_{\Sn}e^{nu} x U d V_{\Sn}\\
=&-n! \int_{\Sn}\left(\frac{\alpha}{(n-1)!}P_n u+1\right)x U d V_{\Sn}\\
=&-\alpha n \int_{\Sn}x U P_n u d V_{\Sn} -n! \int_{\Sn} x U d V_{\Sn}.
\end{align*}
For brevity, we set $F(x)=(1-x^2)^{\frac{n-2}{2}} U=(1-x^2)^{n-1}u'$. By \eqref{formula:GHX} we are ready to deal with
\begin{align*}
&-\alpha n \int_{\Sn}x U P_n u d V_{\Sn}\\
=&-\alpha n \int_{\Sn} x U (-1)^{\frac{n}{2}}[(1-x^2)^{\frac{n-2}{2}}U]^{(n-1)} d V_{\Sn}\\
=& -\alpha n (-1)^{\frac{n}{2}}|\Sp^{n-1}| \int_{-1}^1 F^{(n-1)}(x) x F(x) d x\\
=&-\alpha n |\Sp^{n-1}| \int_{-1}^1  F^{(\frac{n}{2}-1)}(x) [x F(x)]^{(\frac{n}{2})} d x\\
=&-\alpha n |\Sp^{n-1}| \int_{-1}^1  F^{(\frac{n}{2}-1)}(x)\left[\frac{n}{2} F^{(\frac{n}{2}-1)}(x)+x F^{(\frac{n}{2})}(x)\right] d x\\
=&-\alpha n |\Sp^{n-1}| \frac{n-1}{2} \int_{-1}^1  |F^{(\frac{n}{2}-1)}(x)|^2 d x,
\end{align*}
where the third identity follows from integrating by parts for $n/2$ times and using the following fact that since $u$ is a smooth solution by assumption, for $0 \leq k \leq (n-2)/2$ we have
\begin{equation}\label{est:bdry_terms}
F^{(k)}(x)=0 \qquad \mathrm{~~on~~}\quad x=\pm 1.
\end{equation}
The reason for \eqref{est:bdry_terms} is as follows: since $F(x)=(1-x^2)^{\frac{n}{2}}u'(x)$, it suffices to prove that $u^{(k)}$ is bounded near $\pm 1$ for any $k\in\mathbb{N}$. Take the point $-1$ for example, the case of $+1$ is the same. Note that locally we have
$x=-\sqrt{1-|x'|^2}$, where we write $(x_1,\dots,x_n)=:x'$, and $(x_1,\cdots,x_n)$ is a local coordinate system compatible with the differential structure of $\Sn$. So $u(x)$ is a smooth function near the point $(0,\cdots,0,-1)$ if and only if $u(-\sqrt{1-|x'|^2})$ is a smooth function of $x'$ near $x'=0$. Write $\phi(x'):=u(-\sqrt{1-|x'|^2})$, then it is smooth. We shall need the following lemma, which might be well known. For readers' convenience, we include a proof here:

\begin{lemma}
    Let $\phi\in C^\infty(\RR^n)$ be a radially symmetric function, then it is a smooth function of $|x|^2$. To be precise, we have $\phi(x)=v(|x|^2)$, where $v\in C^\infty[0,+\infty)$.
\end{lemma}

\begin{proof}
    Since $\phi$ is radially symmetric, its spherical mean equals itself, namely
    $$\phi(x)=\frac{1}{|\Sp^{n-1}|r^{n-1}}\int_{\partial B_{r}}\phi d\sigma=\frac{1}{|\Sp^{n-1}|}\int_{\Sp^{n-1}}\phi(r\omega) dV_{\Sp^{n-1}}(\omega).$$
We define $\phi(r)$ for $r\in\RR$ by
$$\phi(r):=\frac{1}{|\Sp^{n-1}|}\int_{\Sp^{n-1}}\phi(r\omega) dV_{\Sp^{n-1}}(\omega),$$
then $\phi(r)\in C^\infty(\RR)$ and is even. It suffices to prove that for any $\phi(r)\in C^\infty(\RR)$ which is even, if we set 
$v(t):=\phi(\sqrt{t})$, then $v\in C^\infty[0,+\infty)$.

For this, we use the well known trick to write $\phi(r)=\phi(0)+r\int_0^1\phi'(sr)ds=:\phi(0)+r\phi_1(r)$, where $\phi_1\in C^\infty(\RR)$ is odd. Similarly, we have $\phi_1(r)=r\phi_2(r)$ where $\phi_2\in C^\infty(\RR)$ is even. Inductively, for any $k\in \N$, we can find a polynomial $P_k$ of degree at most $k$ and a smooth even function $\phi_{2k}$ such that
$$\phi(r)=P_k(r^2)+r^{2k}\phi_{2k}(r),$$
from which we get $v(t)=P_k(t)+t^k\phi_{2k}(\sqrt{t})$. For any given $m\in\N$, we can choose $k$ large enough in the above formula. Then it is easy to see that when $t\searrow 0$, any $v^{(j)}$ has a finite limit for $j=1,\cdots,m$. Since $m$ is arbitrary, this shows $v\in C^\infty[0,\infty)$.
\end{proof}

Given this lemma, for $\phi(x')$ as above, we can find a function $v(t)\in C^\infty[0,+\infty)$ such that $\phi(x')=v(|x'|^2)$. Then we have
$$u(x)=v(1-x^2),$$
and hence $u^{(k)}$ is bounded on $[-1,1]$ for any $k\in\N$.

On the other hand, using the fact that $P_n F_2=(n+1)! F_2$ we have
\begin{align*}
\alpha \int_{\Sn} P_n u F_2 d V_{\Sn}=&\alpha \int_{\Sn} u P_n F_2 d V_{\Sn}\\
=& \alpha (n+1)! \int_{\Sn} u F_2 d V_{\Sn}=\frac{\alpha}{2} n! \int_{\Sn} u (-\Delta_{\Sn} F_2) d V_{\Sn}\\
=&\frac{\alpha}{2} n! \int_{\Sn} \langle \nabla u, \nabla F_2\rangle d V_{\Sn}\\
=&-\alpha (n+1)! \int_{\Sn} x U d V_{\Sn}.
\end{align*}

Therefore, putting these facts together we obtain
\begin{equation*}
\alpha |\Sp^{n-1}| \frac{n-1}{2} \int_{-1}^1  |F^{(\frac{n}{2}-1)}(x)|^2 d x=(n-1)!\left((n+1)\alpha-1\right) \int_{\Sn} x U d V_{\Sn}=0.
\end{equation*}

This implies $F^{(\frac{n}{2}-1)}=0$, i.e. 
\begin{align*}
 [(1-x^2)^{n-1} u']^{(\frac{n}{2}-1)}=0, 
\end{align*}
which implies $(1-x^2)^{n-1} u'$ is a polynomial of degree at most $\frac{n}{2}-2$. However, 1 is a zero point of multiplicity at least $n-1>\frac{n}{2}-2$, this implies $u'\equiv 0$. So $u$ is a constant. Going back to equation \eqref{eqn:MFE}, this constant must be 0.    
\end{proof}

\section{Green function rigidity theorems}\label{sec:rigidity}

In this section, we study the spherical rigidity problem for the Green functions. To prove Theorems \ref{thm:rigidity}, we recall the following standard facts in the geometry of Riemannian submanifolds.

 Without loss of generality, we may assume that $Q=0\in\RR^{n+1}$ and  locally $M$ is given as the graph of a function $f:B_\delta(0)\to \RR$ with $f(0)=0$ and $\nabla f(0)=0$. Here $\nabla$ means the gradient with respect to the flat metric on $\RR^n$. Then under the coordinate chart $x=(x_1,\cdots, x_n)$, the induced metric on $M$ is  $g_{\alpha\beta}=\delta_{\alpha\beta}+f_\alpha f_\beta$, where $f_\alpha$ means $\frac{\partial f}{\partial x_\alpha}$. It is direct to check that the second fundamental form is given by $$II=\sum_{\alpha,\beta}h_{\alpha\beta}dx_\alpha\otimes dx_\beta=\frac{1}{\sqrt{1+|\nabla f|^2}}\sum_{\alpha,\beta}f_{\alpha\beta}dx_\alpha\otimes dx_\beta$$ and the mean curvature is $$H=\mathrm{tr}_g(II)=\frac{1}{\sqrt{1+|\nabla f|^2}}g^{\alpha \beta}f_{\alpha \beta}.$$  Denote by $\rho:=\|P-Q\|^2=|x|^2+f^2(x)$, then  a direct computation yields
    $$|\nabla^g\rho|_g^2=4\rho-\frac{4(f-x\cdot\nabla f)^2}{1+|\nabla f|^2},\quad \rho_{,\alpha\beta}=2g_{\alpha\beta}+\frac{2f_{\alpha\beta}(f-x\cdot\nabla f)}{1+|\nabla f|^2}$$
    and 
    $$\Delta_g \rho=2n+\frac{2(f-x\cdot\nabla f)}{1+|\nabla f|^2}g^{\alpha\beta}f_{\alpha\beta}=2n+\frac{2(f-x\cdot\nabla f)}{\sqrt{1+|\nabla f|^2}}H.$$
    Let
    \begin{equation}\label{eqn:eta}
    \eta=\frac{f-x\cdot\nabla f}{\sqrt{1+|\nabla f|^2}},
    \end{equation}
    which is the normal component of the position vector $(x,f(x))$, and usually called a `support function'. 
    Then the above identities become
    \begin{equation}\label{eqn:nabla_lap_rho}
      |\nabla^g\rho|_g^2=4\rho-4\eta^2,\quad \rho_{,\alpha\beta}=2g_{\alpha\beta}+2\eta h_{\alpha\beta},\quad \Delta_g\rho=2n+2\eta H.  
    \end{equation}
    
    We shall use the following lemma:
    \begin{lemma}\label{lem:lim}
       Let $v\in \RR^n$ be a unit vector, and let $\gamma(t):=(tv,f(tv))$ be the corresponding smooth curve on the hypersurface $(M,g)$. Then we have
       \begin{equation}\label{eqn:lim_eta}
           \lim_{t\to 0}\frac{\eta}{\rho}\big(\gamma(t)\big)=-\frac{1}{2}\sum_{\alpha,\beta} f_{\alpha\beta}(0)v_\alpha v_\beta=-\frac{1}{2}II(v,v),
       \end{equation}
       \begin{equation}\label{eqn:lim_nabla_eta}
           \lim_{t\to 0}\frac{\nabla^g\rho\cdot\nabla^g(f-x\cdot\nabla f)}{\rho}(\gamma(t))=-2II(v,v),
       \end{equation}
       \begin{equation}\label{eqn:lim_nabla_eta_sqr}
          \lim_{t\to 0}\frac{|\nabla^g(f-x\cdot\nabla f)|^2_g}{\rho}(\gamma(t))=\sum_{\alpha,\beta,\mu}f_{\mu\alpha}(0)f_{\mu\beta}(0)v_\alpha v_\beta=(II)^2(v,v),
       \end{equation}
       where we always identify\footnote{Note that at $Q$, it follows from our assumption $\nabla f(0)=0$ that $g_{\alpha\beta}=\delta_{\alpha\beta}$.} $v\in\RR^n$ with the unit tangent vector $\sum_\alpha v_\alpha\frac{\partial}{\partial x_\alpha}\in T_Q M$. 
       
    \end{lemma}

    \begin{proof}
        Since $f(tv)=\frac{t^2}{2}\sum_{\alpha,\beta}f_{\alpha\beta}(0)v_\alpha v_\beta+o(t^2)$, and
        $$f_\alpha(tv)=t\sum_\beta f_{\alpha\beta}(0)v_\beta+o(t), $$
        we obtain
        $$(x\cdot\nabla f)(tv)=\sum_\alpha tv_\alpha f_\alpha(tv)=t^2\sum_{\alpha,\beta}f_{\alpha\beta}(0)v_\alpha v_\beta+o(t^2)$$
        and 
        $$|\nabla f|^2(tv)=O(t^2).$$
        So we obtain
        \begin{align*}
          \lim_{t\to 0}\frac{\eta}{\rho}\big(\gamma(t)\big)&=\lim_{t\to 0}\frac{-\frac{t^2}{2}\sum_{\alpha,\beta}f_{\alpha\beta}(0)v_\alpha v_\beta+o(t^2)}{t^2+o(t^2)}=-\frac{1}{2}\sum_{\alpha,\beta}f_{\alpha\beta}(0)v_\alpha v_\beta. 
        \end{align*}
        Also, we have
        $$\partial_\alpha\rho (tv)=2tv_\alpha-2f(tv)f_\alpha(tv)=2tv_\alpha+O(t^2),$$
        $$\partial_\alpha(f-x\cdot\nabla f)(tv)=(-\sum_\beta x_\beta f_{\alpha\beta})(tv)=-t\sum_\beta v_\beta f_{\alpha\beta}(0)+O(t^2),$$
        and 
        $$g^{\alpha\beta}(\gamma(t))=\Big(\delta_{\alpha\beta}-\frac{f_\alpha f_\beta}{1+|\nabla f|^2}\Big)(\gamma(t))=\delta_{\alpha\beta}+O(t^2),$$
        from which we easily obtain \eqref{eqn:lim_nabla_eta} and \eqref{eqn:lim_nabla_eta_sqr}.
    \end{proof}

\subsection{Laplacian on surfaces}\label{subsec:rigidity_surface}

\begin{proof}[Proof of Theorem \ref{thm:rigidity}(1)]
     Recall that $\rho=\|P-Q\|^2$, so $G(P,Q)=-\frac{1}{4\pi}\log\rho+C$. Now the equation for $G$ becomes
    $$-\Delta_g (-\frac{1}{4\pi}\log\rho+C)=-\frac{1}{V},$$
    so we get
    $$\Delta_g\log\rho=-c$$
    for some positive constant $c=\frac{4\pi}{V}$. Direct computation together with \eqref{eqn:nabla_lap_rho} gives
    \begin{align}\label{eqn:Green_surface}
        -c=&\Delta_g\log\rho=\frac{\Delta_g\rho}{\rho}-\frac{|\nabla^g\rho|_g^2}{\rho^2}\nonumber\\
        =&\frac{1}{\rho}\Big(4+2\eta H\Big)-\frac{1}{\rho^2}\Big(4\rho-4\eta^2\Big)=\frac{2\eta H}{\rho}+\frac{4\eta^2}{\rho^2}.
    \end{align}
    Taking limit along the curve $\gamma(t)$ with $ \gamma'(0)=v$, $|v|=1$ and using \eqref{eqn:lim_eta}, we get
    $$-H(Q)II(v,v)+II(v,v)^2=-c.$$
In particular, if $\sum_\beta v_\beta\frac{\partial}{\partial x_\beta}$ is in the principal direction at $Q$ with principal curvature $\kappa_\alpha$, then we get
    $$-\kappa_\alpha H(Q)+\kappa_\alpha^2=-c,\quad \alpha=1,2.$$
    Summing over $\alpha$, we obtain, at $Q$:
    $$2K(Q)=H(Q)^2-|II(Q)|^2=2c.$$
This yields that the Gauss curvature of $(M,g)$ is the positive constant $c$. Consequently, $M$ must be isometric to a round sphere and hence also congruent to a round sphere.    
\end{proof}

\subsection{Conformal Laplacian}\label{subsec:conf_lap}
It is known in \cite{Lee-Parker} that the Green function of conformal Laplacian exists whenever $(M,g)$ is of positive Yamabe constant.
\begin{proof}[Proof of Theorem \ref{thm:rigidity}(2):]
The key point is the following lemma:

\begin{lemma}\label{lem:umbilical}
    Fix $Q \in M$. If $P_{2,P}^g \|P-Q\|^{2-n}=0$ for all $P\in M, P\neq Q$, then $Q$ is an umbilical point.
\end{lemma}

\begin{proof}
 We use the previous conventions. Recall that by the Gauss equation we have $R_g=H^2-|II|^2$. Then we get
    \begin{equation*}
       P_2^g \rho^{1-\frac{n}{2}}=(n-2)\rho^{1-\frac{n}{2}}\Big[\frac{n\eta^2}{\rho^2}+\frac{\eta H}{\rho}+\frac{H^2-|II|^2}{4(n-1)}\Big].
    \end{equation*}
    By our assumption that $P_2^g \rho^{1-\frac{n}{2}}=0$ for all $P \neq Q$, we arrive at the following equation
    \begin{equation}\label{eqn:green_eqn}
       \frac{n\eta^2}{\rho^2}+\frac{\eta H}{\rho}+\frac{H^2-|II|^2}{4(n-1)}=0.
    \end{equation}
Again, taking limit along  $\gamma(t):=(tv, f(tv))$  with $|v|=1$, together with \eqref{eqn:lim_eta} we have
    \begin{align*}
        \frac{n}{4}II(v,v)^2-\frac{H(Q)}{2}II(v,v)+\frac{H^2-|II|^2}{4(n-1)}(Q)=0.
    \end{align*}
    As in the surface case, if $\sum_\beta v_\beta\frac{\partial}{\partial x_\beta}$ is in the principal direction at $Q$ with principal curvature $\kappa_\alpha$, then we get
    $$\frac{n}{4}\kappa^2_\alpha-\frac{1}{2}\kappa_\alpha H(Q)+\frac{H^2-|II|^2}{4(n-1)}(Q)=0,\qquad \forall~ \alpha.$$
    Summing over $\alpha$, we obtain, at $Q$:
    $$\frac{n}{4}|II|^2-\frac{1}{2}H^2+\frac{n(H^2-|II|^2)}{4(n-1)}=0,$$
    or equivalently $H^2(Q)=n|II|^2(Q)$. That is, $(\sum_\alpha\kappa_\alpha)^2=n\sum_\alpha\kappa_\alpha^2$. This forces $\kappa_1=\dots=\kappa_n$ and hence $Q$ is umbilical.        
\end{proof}

Now the theorem follows from the fact that a closed embedded umbilical hypersurface in $\RR^{n+1}$ is isometric to a round sphere. In fact, we already have $II=\frac{H}{n}g$. Let $\{e_\alpha\}$ be an orthonormal frame and $\{\omega^\alpha\}$ be its dual coframe. Write $II=h_{\alpha\beta}\omega^\alpha\otimes\omega^\beta$, then taking covariant derivatives to show $\frac{H_{,\gamma}}{n}g_{\alpha\beta}=\nabla_\gamma h_{\alpha\beta}$, and hence
\begin{align*}
  \frac{H_{,\alpha}}{n}&=g^{\beta\gamma}\frac{H_{,\gamma}}{n}g_{\alpha\beta}=g^{\beta\gamma}\nabla_\gamma h_{\alpha\beta}=g^{\beta\gamma}\nabla_\alpha h_{\beta\gamma}=H_{,\alpha}, 
\end{align*}
where we have used the Codazzi equation. This implies $H_{,\alpha}=0,\forall \alpha$ and hence $H$ is constant. Thus the Alexandrov theorem \cite{Alexandrov} yields that $M$ must be a round sphere.  
\end{proof}

\subsection{Paneitz operator}

We now consider the Paneitz operator, the fourth order conformally covariant operator introduced by Paneitz \cite{Paneitz}, 
\begin{equation}\label{def:Paneitz_operator}
        P_4^g \psi=\Delta^2_g\psi+\frac{4}{n-2}\div\big(\sum_\alpha \mathrm{Ric}(\nabla^g\psi,e_\alpha)e_\alpha\big)-\frac{n^2-4n+8}{2(n-1)(n-2)}\div(R_g\nabla^g\psi)+\frac{n-4}{2}Q_g\psi,
        \end{equation}
        where $\{e_\alpha\}$ is a local orthonormal frame, and the Branson $Q$-curvature is defined by
        $$Q_g=-\frac{1}{2(n-1)}\Delta_g R_g+\frac{n^3-4n^2+16n-16} {8(n-1)^2(n-2)^2} R_g^2-\frac{2}{(n-2)^2}|\mathrm{Ric}|_g^2.$$

We start with the non-critical case. 

\begin{proof}[Proof of Theorem \ref{thm:rigidity}(3):]
    We follow the same idea as Theorem \ref{thm:rigidity}(2). The key point is also the following lemma:
    \begin{lemma}
       If $P_4^g \|P-Q\|^{4-n}=0$ for all $P\neq Q$, then $Q$ is an umbilical point.
    \end{lemma}
    \begin{proof}
        The proof follows the same line as in Section \ref{subsec:conf_lap}. Let $\psi(\rho):=\rho^{2-\frac{n}{2}}$. 
        Then as before, we take an arbitrary unit vector $v\in\RR^n$ and study the limit of $\rho^{\frac{n}{2}-1}P_4^g\psi(\rho)$ along  $\gamma(t):=(tv,f(tv))$  as $t\to 0$.

         Clearly, $\rho^{\frac{n}{2}-1}Q_g\psi(\rho)=Q_g\rho\to 0$ as $t\to 0$. It suffices to consider the first three terms in the above formula for $P_4^g \psi(\rho)$. Since
        \begin{equation}
          \Delta_g \psi(\rho)=(n-4)\Big[-2\rho^{1-\frac{n}{2}}-\rho^{1-\frac{n}{2}}\eta H-(n-2)\rho^{-\frac{n}{2}} \eta^2\Big],
        \end{equation}
        we can write $\Delta_g^2\psi(\rho)=(n-4)(A+B+C)$, where
        \begin{align*}
        A:=&-2\Delta_g \rho^{1-\frac{n}{2}}=2(n-2)\rho^{1-\frac{n}{2}}\Big(\frac{n\eta^2}{\rho^2}+\frac{\eta H}{\rho}\Big),\\
        B:=&-\Delta_g\Big(\rho^{1-\frac{n}{2}}\eta H\Big),\\
        C:=&-(n-2)\Delta_g\Big(\rho^{-\frac{n}{2}} \eta^2\Big).
        \end{align*}
        By \eqref{eqn:lim_eta}, we have
        \begin{equation}\label{eqn:A_lim}
            \lim_{t\to 0}\rho^{\frac{n}{2}-1}A(\gamma(t))=2(n-2)\Big[\frac{n}{4}II(v,v)^2-\frac{1}{2}H(Q)II(v,v)\Big].
        \end{equation}
        For $B$, we have
        \begin{align*}
            B&=-\Delta_g(\rho^{1-\frac{n}{2}})\eta H-\rho^{1-\frac{n}{2}}\Delta_g(\eta H)+(n-2)\rho^{-\frac{n}{2}}\nabla^g\rho\cdot\nabla^g(\eta H).
        \end{align*}
        Since $\eta H=O(|x|^2)$, we conclude that
        \begin{align*}
            \lim_{t\to 0}\rho^{\frac{n}{2}-1}B(\gamma(t))&=0-H(Q)\Delta_g(f-x\cdot\nabla f)(Q)\\
            &\quad +(n-2)H(Q)\lim_{t\to 0}\frac{\nabla^g\rho\cdot\nabla^g(f-x\cdot\nabla f)}{\rho}(\gamma(t)).
        \end{align*}
        On one hand, we have
        \begin{align*}
            \Delta_g(f-x\cdot\nabla f)(Q)&=g^{\alpha\beta}(0)\frac{\partial^2}{\partial x_\alpha\partial x_\beta}(f-x\cdot\nabla f)(0)\\
            &=\delta^{\alpha\beta}(-f_{\alpha\beta}(0))=-H(Q).
        \end{align*}
        Combining with \eqref{eqn:lim_nabla_eta}, we obtain
        \begin{equation}\label{eqn:B_lim}
             \lim_{t\to 0}\rho^{\frac{n}{2}-1}B(\gamma(t))=H^2(Q)-2(n-2)H(Q)II(v,v).
        \end{equation}
        For term $C$, we have
        \begin{align*}
             \frac{\rho^{\frac{n}{2}-1}C}{n-2}&=-\rho^{\frac{n}{2}-1}\Delta_g(\rho^{-\frac{n}{2}})\eta^2-\rho^{-1}\Delta_g(\eta^2)+n\rho^{-2}\nabla^g\rho\cdot\nabla^g(\eta^2).
        \end{align*}
        Since
        \begin{align*}
            \Delta_g\rho^{-\frac{n}{2}}&=-\frac{n}{2}\rho^{-\frac{n}{2}-1}\Delta_g\rho+\frac{n(n+2)}{4}\rho^{-\frac{n}{2}-2}|\nabla^g\rho|_g^2\\
            &=-\frac{n}{2}\rho^{-\frac{n}{2}-1}\Big(2n+2\eta H\Big)+\frac{n(n+2)}{4}\rho^{-\frac{n}{2}-2}\Big(4\rho-4\eta^2\Big)\\
            &=2n\rho^{-\frac{n}{2}-1}-n\rho^{-\frac{n}{2}-1}\eta H-n(n+2)\rho^{-\frac{n}{2}-2}\eta^2,
        \end{align*}
        we have $\lim_{t\to 0}\rho^{\frac{n}{2}+1}\Delta_g\rho^{-\frac{n}{2}}=2n$,
        and hence
        $$-\lim_{t\to 0}\rho^{\frac{n}{2}-1}\Delta_g(\rho^{-\frac{n}{2}})\eta^2(\gamma(t))=-2n\cdot\frac{1}{4}II(v,v)^2=-\frac{n}{2}II(v,v)^2.$$
        Also, using \eqref{eqn:lim_nabla_eta_sqr} and \eqref{eqn:lim_nabla_eta}, we have
        \begin{align*}
            -\lim_{t\to 0}\rho^{-1}\Delta_g(\eta^2)(\gamma(t))&=-\lim_{t\to 0}\frac{\Delta_g(f-x\cdot\nabla f)^2}{\rho}(\gamma(t))\\
            &=-\lim_{t\to 0}\frac{2(f-x\cdot\nabla f)\Delta_g(f-x\cdot\nabla f)+2|\nabla(f-x\cdot\nabla f)|_g^2}{\rho}\\
            &=-H(Q)II(v,v)-2\sum_{\alpha,\beta,\mu}f_{\mu\alpha}(0)f_{\mu\beta}(0)v_\alpha v_\beta\\
            &=-H(Q)II(v,v)-2(II)^2(v,v)
        \end{align*}
        and
        \begin{align*}
            \lim_{t\to 0}n\rho^{-2}\nabla^g\rho\cdot\nabla^g\Big(\eta^2\Big)(\gamma(t))&=\lim_{t\to 0}\frac{2n(f-x\cdot \nabla f)}{\rho}\frac{\nabla^g\rho\cdot\nabla^g(f-x\cdot \nabla f)}{\rho}\\
            &=2nII(v,v)^2.
        \end{align*}
        So we arrive at
        \begin{equation}\label{eqn:C_lim}
           \frac{1}{n-2}\lim_{t\to 0}\rho^{\frac{n}{2}-1}C(\gamma(t))=\frac{3n}{2}II(v,v)^2-2(II)^2(v,v)-H(Q)II(v,v). 
        \end{equation}
        
       Therefore, combining \eqref{eqn:A_lim}, \eqref{eqn:B_lim} and \eqref{eqn:C_lim} we obtain
       \begin{equation}\label{eqn:bilap_lim}
           \lim_{t\to 0}\frac{\rho^{\frac{n}{2}-1}\Delta_g^2\psi(\rho)}{(n-2)(n-4)}(\gamma(t))=2nII(v,v)^2-2(II)^2(v,v)-4H(Q)II(v,v)+\frac{H(Q)^2}{n-2}.
       \end{equation}

       We next compute the limit of the remaining two divergence terms, which are
       \begin{align*}
         &\quad  \frac{4}{n-2}\div\big(\sum_\alpha \mathrm{Ric}(\nabla^g\psi(\rho),e_\alpha)e_\alpha\big)-\frac{n^2-4n+8}{2(n-1)(n-2)}\div(R_g\nabla^g\psi(\rho))\\
         &=\frac{4}{n-2}\big(R^\alpha_\beta\psi_\alpha\big)_,^\beta-\frac{n^2-4n+8}{2(n-1)(n-2)}\big(\nabla^gR_g\cdot\nabla^g\psi+R_g\Delta_g\psi\big)\\
         &=\frac{4}{n-2}\big(R^\alpha_\beta\psi'(\rho)\rho_\alpha\big)_,^\beta-\frac{n^2-4n+8}{2(n-1)(n-2)}\big(\psi'(\rho)\nabla^gR_g\cdot\nabla^g\rho+R_g\Delta_g\psi\big)\\
         &=\frac{4}{n-2}\Big(R^{\alpha\ \beta}_{\beta,}\psi'(\rho)\rho_\alpha+\psi''(\rho)\mathrm{Ric}(\nabla^g\rho,\nabla^g\rho)+\psi'(\rho)\langle \mathrm{Ric},\mathrm{Hess}_g\rho\rangle\Big)\\
         &\quad -\frac{n^2-4n+8}{2(n-1)(n-2)}\big(\psi'(\rho)\nabla^gR_g\cdot\nabla^g\rho+R_g\Delta_g\psi\big)\\
         &=\Big(\frac{2}{n-2}-\frac{n^2-4n+8}{2(n-1)(n-2)}\Big)\psi'(\rho)\nabla^gR_g\cdot\nabla^g\rho+\frac{4\psi'(\rho)}{n-2}\langle \mathrm{Ric},\mathrm{Hess}_g\rho\rangle \\
         &\quad+
         \frac{4\psi''(\rho)}{n-2}\mathrm{Ric}(\nabla^g\rho,\nabla^g\rho)-\frac{n^2-4n+8}{2(n-1)(n-2)}R_g\Delta_g\psi(\rho)\\
         &=-\frac{n-6}{2(n-1)}\psi'(\rho)\nabla^gR_g\cdot\nabla^g\rho+\frac{4\psi'(\rho)}{n-2}\langle \mathrm{Ric},\mathrm{Hess}_g\rho\rangle+
         \frac{4\psi''(\rho)}{n-2}\mathrm{Ric}(\nabla^g\rho,\nabla^g\rho)\\
         &\quad-\frac{n^2-4n+8}{2(n-1)(n-2)}R_g\Delta_g \psi(\rho).
       \end{align*}
       This implies that
       \begin{align*}
           &\lim_{t\to 0}\rho^{\frac{n}{2}-1}\Big( \frac{4}{n-2}\div\big(\sum_\alpha \mathrm{Ric}(\nabla^g\psi(\rho),e_\alpha)e_\alpha\big)-\frac{n^2-4n+8}{2(n-1)(n-2)}\div(R_g\nabla^g\psi(\rho))\Big)(\gamma(t))\\
           &=0+\frac{4}{n-2}(-\frac{n-4}{2})\langle \mathrm{Ric},\mathrm{Hess}_g\rho\rangle(0)+(n-4)\lim_{t\to 0}\frac{\mathrm{Ric}(\nabla^g\rho,\nabla^g\rho)}{\rho}\\
           &\quad -\frac{n^2-4n+8}{2(n-1)(n-2)}R_g(Q)\lim_{t\to 0}\big(-2(n-4)+O(|x|^2)\big)(\gamma(t))\\
           &=-\frac{2(n-4)}{n-2}R_g(Q)+4(n-4)\mathrm{Ric}(v,v)+\frac{(n-4)(n^2-4n+8)}{(n-1)(n-2)}R_g(Q)\\
           &=\frac{(n-4)(n-6)}{n-1}R_g(Q)+4(n-4)\mathrm{Ric}(v,v).
       \end{align*}
       Combining with \eqref{eqn:bilap_lim}, we obtain
       \begin{align*}
      &2nII(v,v)^2-2(II)^2(v,v)-4H(Q)II(v,v)\\
      &+\frac{H(Q)^2}{n-2}+\frac{n-6}{(n-1)(n-2)}R_g(Q)+\frac{4}{n-2}\mathrm{Ric}(v,v)=0
      \end{align*}
      for any unit vector $v$. Replacing $v$ by every principal direction in an orthonormal basis of $T_QM$, and taking trace, we obtain
      $$2(n-1)|II(Q)|^2-4H(Q)^2+\frac{n}{n-2}H(Q)^2+\Big(\frac{n(n-6)}{(n-1)(n-2)}+\frac{4}{n-2}\Big)R_g(Q)=0,$$
      that is,
      $$2(n-1)|II(Q)|^2-\frac{3n-8}{n-2}H(Q)^2+\frac{n^2-2n-4}{(n-1)(n-2)}R_g(Q)=0.$$
      Using the fact that $R_g=H^2-|II|^2$, we obtain
      $$\frac{n(2n^2-9n+12)}{(n-1)(n-2)}|II(Q)|^2-\frac{2n^2-9n+12}{(n-1)(n-2)}H(Q)^2=0$$
      and hence $H(Q)^2=n|II(Q)|^2$, which implies that $Q$ is umbilical.
    \end{proof}
The remaining part of proof is identical to that of Theorem \ref{thm:rigidity}(2).
\end{proof}

Now we prove the corresponding theorem for $P_4^g$ on a closed 4-manifold.

\begin{proof}[Proof of Theorem \ref{thm:rigidity}(4)]
    The proof is almost the same as Theorem \ref{thm:rigidity}(3), so we shall be brief here. Recall that in dimension four, the Paneitz operator becomes
    $$P_4^g u=\Delta_g^2 u-\frac{2}{3}R_g\Delta_g u+2\langle \mathrm{Ric}, \mathrm{Hess}_g u\rangle+\frac{1}{3}\nabla^g R_g\cdot\nabla^g u.$$
    Write $\phi(\rho)=\log\rho$, then it follows from our assumption that
    $$P_4^g \phi(\rho)=c_0$$
    for a constant $c_0$ when $\rho\neq 0$. Then we have
    $$\lim_{t\to 0}\big[\rho P_4^g \phi(\rho)\big](\gamma(t))=0$$
    for  $\gamma(t)$ as before.
    
    Now we compute the above limit directly.
    First we have
    $$\Delta_g\phi(\rho)=\frac{\Delta_g\rho}{\rho}-\frac{|\nabla^g\rho|_g^2}{\rho^2}=\frac{4}{\rho}+\frac{2H\eta}{\rho}+\frac{4\eta^2}{\rho^2}.$$
    Then we compute $\Delta_g^2\phi(\rho)$ term by term:
    \begin{align*}      
    \Delta_g\Big(\frac{4}{\rho}\Big)&=4\Big(-\frac{\Delta_g\rho}{\rho^2}+\frac{2|\nabla^g\rho|_g^2}{\rho^3}\Big)\\
    &=-\frac{4}{\rho^2}(8+2H\eta)+\frac{8}{\rho^3}(4\rho-4\eta^2)\\
    &=-\frac{8H\eta}{\rho^2}-\frac{32\eta^2}{\rho^3};\\
        \Delta_g\Big(\frac{2H\eta}{\rho}\Big)&=\Delta_g\frac{1}{\rho}(2H\eta)+\frac{1}{\rho}\Delta_g(2H\eta)+2\nabla^g\frac{1}{\rho}\cdot\nabla^g(2H\eta)\\
        &=2H\eta\Big(-\frac{2H\eta}{\rho^2}-\frac{8\eta^2}{\rho^3}\Big)+\frac{1}{\rho}\Delta_g(2H\eta)-\frac{4}{\rho^2}\nabla^g\rho\cdot\nabla^g(H\eta)\\
        &=-\frac{4H^2\eta^2}{\rho^2}-\frac{16H\eta^3}{\rho^3}+\frac{1}{\rho}\Delta_g(2H\eta)-\frac{4}{\rho^2}\nabla^g\rho\cdot\nabla^g(H\eta)
    \end{align*}
    and
    \begin{align*}
        \Delta_g\Big(\frac{4\eta^2}{\rho^2}\Big)&=4\Delta_g\frac{1}{\rho^2}\eta^2+\frac{4}{\rho^2}\Delta_g\eta^2+8\nabla^g\frac{1}{\rho^2}\cdot\nabla^g\eta^2\\
        &=-\frac{8\eta^2}{\rho^3}\Delta_g\rho+\frac{24\eta^2}{\rho^4}|\nabla^g\rho|_g^2+\frac{8}{\rho^2}(\eta\Delta_g\eta+|\nabla\eta|_g^2)-\frac{32\eta}{\rho^3}\nabla^g\rho\cdot\nabla^g\eta.
    \end{align*}

 Based on these formulae, we compute
    \begin{align*}
        \lim_{t\to 0}\rho\Delta_g\Big(\frac{4}{\rho}\Big)(\gamma(t))&=-8H(Q)(-\frac{1}{2}II(v,v))-32(-\frac{1}{2}II(v,v))^2\\
        &=4H(Q)II(v,v)-8II(v,v)^2;\\
        \lim_{t\to 0}\rho\Delta_g\Big(\frac{2H\eta}{\rho}\Big)(\gamma(t))&=0+0+2H(Q)\Delta_g\eta(0)-4H(Q)\lim_{t\to 0}\frac{\nabla^g\rho\cdot\nabla^g\eta}{\rho}\\
        &=2H(Q)(-H(Q))-4H(Q)(-2II(v,v))\\
        &=-2H(Q)^2+8H(Q)II(v,v)
    \end{align*}
    and
     \begin{align*}
        \lim_{t\to 0}\rho\Delta_g\Big(\frac{4\eta^2}{\rho^2}\Big)(\gamma(t))&=-8\cdot\frac{1}{4}II(v,v)^2\Delta_g\rho(0)+24\cdot\frac{1}{4}II(v,v)^2\cdot 4\\
        &\quad+8\cdot(-\frac{1}{2}II(v,v))\Delta_g\eta(0)+8(II)^2(v,v)\\
        &\quad-32(-\frac{1}{2}II(v,v))(-2II(v,v))\\
        &=-24II(v,v)^2+8(II)^2(v,v)+4H(Q)II(v,v).
    \end{align*}
    Combining the above formulae, we obtain
    \begin{equation}\label{eqn:bilap_lim4d}
        \lim_{t\to 0}\rho\Delta_g^2\phi(\rho)(\gamma(t))=-32II(v,v)^2+8(II)^2(v,v)+16H(Q)II(v,v)-2H(Q)^2.
    \end{equation}
    
    Next, we compute
    \begin{equation}\label{eqn:R_laplim}
     \lim_{t\to 0}\rho\Big(-\frac{2}{3}R_g\Delta_g\phi\Big)(\gamma(t))=-\frac{2}{3}R_g(Q)\cdot 4=-\frac{8}{3}R_g(Q),   
    \end{equation}
    and
    \begin{equation}\label{eqn:nablaRlim}
        \lim_{t\to 0}\rho\cdot\frac{1}{3}\nabla^gR_g\cdot\nabla^g\phi=\frac{1}{3}\nabla^gR_g(0)\cdot\nabla^g\rho(0)=0.
    \end{equation}
    
    Finally, we have
    \begin{align}\label{eqn:Riclim}
        \lim_{t\to 0}\rho\cdot 2\langle \mathrm{Ric},\mathrm{Hess}_g\phi\rangle(\gamma(t))=&\lim_{t\to 0}2\rho \Big(\frac{\langle \mathrm{Ric},\mathrm{Hess}_g\rho\rangle}{\rho}-\frac{\mathrm{Ric}(\nabla^g \rho, \nabla^g \rho)}{\rho^2}\Big)\nonumber\\
        =&4R_g(Q)-8\mathrm{Ric}(v,v).
    \end{align}
    
    Combining \eqref{eqn:bilap_lim4d}, \eqref{eqn:R_laplim}, \eqref{eqn:nablaRlim} and \eqref{eqn:Riclim}, we obtain
    \begin{align*}
         -32II(v,v)^2+8(II)^2(v,v)+16H(Q)II(v,v)-2H(Q)^2+\frac{4}{3}R_g(Q)-8\mathrm{Ric}(v,v)=0.
    \end{align*}
    Taking trace, we obtain
    \begin{equation*}
        -24|II(Q)|^2+8H(Q)^2-\frac{8}{3}R_g(Q)=0.
    \end{equation*}
    Since $R_g=H^2-|II|^2$, we again obtain $H(Q)^2=4|II(Q)|^2$ and hence $Q$ is umbilical. Then we follow the same route as before.
\end{proof}

\section{Strong rigidity in low dimensions}\label{sec:strong rigidity}

With the help of the Positive Mass Theorem we now prove the strong rigidity conjecture for conformal Laplacian when $n=3,4,5$.

\begin{proof}[Proof of Theorem \ref{thm:rigidity_PMT}:] We first prove that $(M^n,g)$ is globally conformal to the round sphere $\Sn$.

\begin{lemma}\label{lem:AF coordinates}
    The manifold $(M\setminus\{Q\}, \rho^{-2}g)$ is asymptotically flat of order $2$ with respect to the coordinates $y:=x|x|^{-2}$. Moreover, its scalar curvature vanishes identically.
\end{lemma}

\begin{proof}
    Write $\hat g:=\rho^{-2}g$, then in a neighborhood of $Q$, we have
    \begin{align*}
        \hat g&= (|x|^2+f^2(x))^{-2}\sum_{\alpha,\beta}(\delta_{\alpha\beta}+f_\alpha f_\beta)dx_\alpha\otimes dx_\beta\\
        &=\Big(\frac{1}{|y|^2}+f^2\big(\frac{y}{|y|^2}\big)\Big)^{-2}\sum_{\alpha,\beta}\Big(\delta_{\alpha\beta}+f_\alpha\big(\frac{y}{|y|^2}\big) f_\beta\big(\frac{y}{|y|^2}\big)\Big)\cdot\\
        &\qquad \Big(\frac{dy_\alpha}{|y|^2}-\sum_\mu\frac{2y_\alpha y_\mu}{|y|^4}dy_\mu\Big)\otimes \Big(\frac{dy_\beta}{|y|^2}-\sum_\nu\frac{2y_\beta y_\nu}{|y|^4}dy_\nu\Big)\\
        &=\Big(1+|y|^2f^2\big(\frac{y}{|y|^2}\big)\Big)^{-2}\sum_{\alpha,\beta}\Big(\delta_{\alpha\beta}+f_\alpha\big(\frac{y}{|y|^2}\big) f_\beta\big(\frac{y}{|y|^2}\big)\Big)\cdot\\
        &\qquad \Big(dy_\alpha-\sum_\mu\frac{2y_\alpha y_\mu}{|y|^2}dy_\mu\Big)\otimes \Big(dy_\beta-\sum_\nu\frac{2y_\beta y_\nu}{|y|^2}dy_\nu\Big)\\
        &=:\sum_\alpha dy_\alpha\otimes dy_\alpha+\sum_{\alpha,\beta}h_{\alpha\beta}dy_\alpha\otimes dy_\beta.
    \end{align*}
    Making use of the fact that $f(x)=O(|x|^2)$, $f_\alpha(x)=O(|x|)$, it is not hard to see that $h_{\alpha\beta}=O(|y|^{-2})$, $\partial h_{\alpha\beta}=O(|y|^{-3})$ and $\partial^2 h_{\alpha\beta}=O(|y|^{-4})$. So $(M\setminus\{Q\}, \hat g)$ is asymptotically flat of order $2$ under the asymptotic coordinates $y$.

    Since $\rho^{-2}$ is $G(\cdot,Q)^{\frac{4}{n-2}}$ up to a constant factor, the conformal covariance of conformal Laplacian shows that the scalar curvature of $\hat g$ is identically $0$.
\end{proof}

It is well known that if the order $\tau$ of an asymptotically flat $n$-manifold satisfies $\tau>\frac{n-2}{2}$, then the mass is well defined and independent of the asymptotic coordinates. Now $\tau=2$ in our case, so the mass is well defined when $n=3,4,5$.

\begin{lemma}\label{lem:AF mass}
    When $n=3,4,5$, the mass of $(M\setminus\{Q\},\hat g)$ is $0$.
\end{lemma}

\begin{proof}
    By Lemma \ref{lem:umbilical}, $Q$ is umbilical. Simplify $H$ for $H(Q)$ and $A$ for the homogeneous polynoimial $\frac{1}{6}\sum_{\alpha\beta\gamma}f_{\alpha\beta\gamma}(0)x_\alpha x_\beta x_\gamma$. Then we have $f_{\alpha\beta}(0)=\frac{H(Q)}{n}\delta_{\alpha\beta}$, whence 
    $$f(x)=\frac{1}{2}\sum_{\alpha,\beta}f_{\alpha\beta}(0)x_\alpha x_\beta+A+O(|x|^4)=\frac{H}{2n}|x|^2+A+O(|x|^4)$$
    and
    $$f_\alpha(x)=\sum_\beta f_{\alpha\beta}(0) x_\beta+\pa_\alpha A+O(|x|^3)=\frac{H}{n}x_\alpha+\pa_\alpha A+O(|x|^3).$$
    Consequently,
    \begin{align*}
       \hat g_{\alpha\beta}=&\Big(1-\frac{H^2}{2n^2}\frac{1}{|y|^2}-\frac{2HA}{n}+O(|y|^{-4})\Big)\\
       &\cdot \Big(\delta_{\alpha\beta}+\frac{H^2}{n^2}\frac{y_\alpha y_\beta}{|y|^4}-\frac{H}{n}(y^\alpha\frac{\pa A}{\pa y^\beta}+y^\beta\frac{\pa A}{\pa y^\alpha})+O(|y|^{-4})\Big),
    \end{align*}
    where $O(|y|^{-4})$ is a term whose $k$-th order derivative decays at the rate of $|y|^{-4-k}$ for any $k\geq 0$. Then we get
    \begin{align}\label{eqn:g_trace}
        \sum_\alpha \hat g_{\alpha\alpha}&=\Big(1-\frac{H^2}{2n^2}\frac{1}{|y|^2}-\frac{2HA}{n}+O(|y|^{-4})\Big)\cdot \Big(n+\frac{H^2}{n^2}\frac{1}{|y|^2}+\frac{6H}{n}A+O(|y|^{-4})\Big),
    \end{align}
    where we used Euler's formula for the homogeneous function $A$. Denote by $r=|y|$ and $\frac{\pa}{\pa r}=\sum_\alpha\frac{y^\alpha}{|y|}\frac{\pa}{\pa y^\alpha}$, and $\hat g_{rr}:=\hat g(\frac{\pa}{\pa r},\frac{\pa}{\pa r})$. Then we have, by similar computation, that
    \begin{align}\label{eqn:g_rr}
        \hat g_{rr}&=\Big(1-\frac{H^2}{2n^2}\frac{1}{|y|^2}-\frac{2HA}{n}+O(|y|^{-4})\Big)\cdot \Big(1+\frac{H^2}{n^2}\frac{1}{|y|^2}+\frac{6H}{n}A+O(|y|^{-4})\Big).
    \end{align}
    Now recall the following formula of the ADM mass (see, e.g. Lee-Parker \cite[Equation (9.6) on p.79]{Lee-Parker})
    \begin{equation}\label{eqn:mass_formula}
        m(\hat g)=\lim_{r\to\infty}\int_{|y|=r} \Big(\pa_r(\hat g_{rr}-\sum_\alpha\hat g_{\alpha\alpha})+r^{-1}(n\hat g_{rr}-\sum_\alpha \hat g_{\alpha\alpha})\Big)d\sigma.
    \end{equation}
   By \eqref{eqn:g_trace} and \eqref{eqn:g_rr}, we get
   \begin{align*}
      \hat g_{rr}-\sum_\alpha\hat g_{\alpha\alpha}=-(n-1)\Big(1-\frac{H^2}{2n^2}\frac{1}{|y|^2}-\frac{2HA}{n}+O(|y|^{-4})\Big),
   \end{align*}
   and consequently
   \begin{equation}\label{eqn:mass_1st}
     \pa_r(\hat g_{rr}-\sum_\alpha\hat g_{\alpha\alpha})=-(n-1)\Big(\frac{H^2}{n^2}\frac{1}{|y|^3}+\frac{6HA}{n}\frac{1}{|y|}+O(|y|^{-5})\Big).  
   \end{equation}
   Also, we have
   \begin{align}\label{eqn:mass_2nd}
     \nonumber r^{-1}(n\hat g_{rr}-\sum_\alpha \hat g_{\alpha\alpha})&=(n-1)\Big(1-\frac{H^2}{2n^2}\frac{1}{|y|^2}-\frac{2HA}{n}+O(|y|^{-4})\Big)\\ 
     \nonumber&\quad\cdot\Big(\frac{H^2}{n^2}\frac{1}{|y|^3}+\frac{6HA}{n}\frac{1}{|y|}+O(|y|^{-5})\Big)\\
     &=(n-1)\Big(\frac{H^2}{n^2}\frac{1}{|y|^3}+\frac{6HA}{n}\frac{1}{|y|}+O(|y|^{-5})\Big).
   \end{align}
   Combining \eqref{eqn:mass_formula},\eqref{eqn:mass_1st} and \eqref{eqn:mass_2nd}, we obtain
   $$m(\hat g)=\lim_{r\to\infty}\int_{|y|=r}O(|y|^{-5})d\sigma=\lim_{r\to\infty}O(r^{n-6})=0,$$
   where we used the assumption $n=3,4,5$ in the last equality.
    \end{proof}

\begin{remark}\label{rmk:coordinate change}
\begin{enumerate}[(i)]
    \item  In the proof of Lemmas \ref{lem:AF coordinates} and \ref{lem:AF mass} we actually prove the following more general result: \textit{If $M^n$ is an embedded closed hypersurface in $\RR^{n+1}$ and $Q\in M$ is an umbilical point. Let $\rho$ be the function $\|\cdot-Q\|^2$ , then $(M\setminus\{Q\}, \rho^{-2}g)$ is asymptotically flat of order 2, and the mass is 0 when $n=3,4,5$.}
    \item If we employ a further asymptotic coordinate change $z=\exp\Big(-\frac{H^2(Q)}{4n^2 |y|^2}\Big)y$, then $(M\setminus\{Q\},\hat g)$ is asymptotically flat of order 3 with respect to $z$. So the mass is well defined even for $n=6,7$. However, it is difficult  to figure out the exact value of its mass when $n>5$.
\end{enumerate}
\end{remark}

Applying the positive mass theorem relying on the work of Schoen and Yau \cite{SchYau79}, we conclude that $(M\setminus\{Q\},\hat g)$ is isometric to the flat Euclidean space $\RR^n$, and hence $(M,g)$ is conformal to the round sphere $\Sn$.\\ 
    
    Next we prove that $(M,g)$ is in fact {\em isometric} to a round sphere.

Using the inverse map of the isometry $$(M\setminus\{Q\},\hat g)\to (\RR^n, g_{\RR^n}),$$ we obtain a smooth embedding $F: \RR^n \to \RR^{n+1}$ satisfying 
$$F^*g_{\RR^{n+1}}=|F|^4 g_{\RR^n} \qquad \mathrm{~~in~~} \RR^n$$
with the normalization that $\lim_{|y|\to\infty}F(y)=0$. Our purpose is to show that the hypersurface $F$ must be a round sphere minus one point.

Now we use the inversion map in $\RR^{n+1}$, and set 
$$U(y)=\frac{F(y)}{|F(y)|^2}.$$
Then it is direct to check that $U$ satisfies $$\partial_\alpha U \cdot \partial_\beta U= \delta_{\alpha \beta}, \quad 1 \leq \alpha,\beta \leq n$$ in $\RR^n$. This implies that the hypersurace $U$ is a complete flat hypersurface in $\RR^{n+1}$. The Hartman-Nirenberg theorem in \cite{Hartman-Nirenberg} shows that the hypersurace $U$ must be a cylinder over a regular plane curve. Explicitly, (after a rigid motion in $\RR^{n+1}$) the  hypersurface $U$ has a parametrization of the form
$$U(t,z)=(x(t),y(t),z), \qquad t \in \RR,~~z \in \RR^{n-1},$$
where $t \mapsto (x(t),y(t))$ is  a regular plane curve parametrized by arc length $t$. Here we abuse some notations $x,y$ for a moment. We comment that an alternative proof of the above result for $n=2$ is due to Massey \cite{Massey}.

Under the coordinates $(t,z) \in \RR^n$, the hypersurface $F$ becomes
$$F(t,z)=\frac{(x(t),y(t),z)}{x(t)^2+y(t)^2+|z|^2}.$$
We next compute the principal curvatures of the hypersurface $F$. To this end, a direct computation yields
\begin{align*}
\pa_t F=&\frac{(x',y',0)}{x^2+y^2+|z|^2}-\frac{2(xx'+yy')}{x^2+y^2+|z|^2}F,\\
\pa_{z_i}F=&\frac{\mathbf{e}_{i+2}}{x^2+y^2+|z|^2}-\frac{2z_i}{x^2+y^2+|z|^2} F 
\end{align*}
with the property that $\langle \pa_t F,\pa_{z_i}F \rangle=0$ for all $1\leq i \leq n-1$.
It is not hard to see that the unit normal vector field is given by
$$\mathbf{n}=(-y',x',0)+2(xy'-x'y)F.$$

Notice that
$$\pa_{z_i}\mathbf{n}=2(xy'-x'y)\pa_{z_i}F ,\qquad 1 \leq i \leq n-1,$$
which implies that $\lambda=-2(xy'-x'y)$ is a principal curvature of multiplicity at least $n-1$.
So, $\pa_t F$ must be a principal direction of another principal curvature $\mu$, that is,
$\mu \pa_t F=- \pa_t \mathbf{n}$. Now
\begin{align*}
\pa_t \mathbf{n}=(-y'',x'',0)+2(xy'-x'y)' F +2(xy'-x'y)\pa_t F.
\end{align*}
This in turn implies that $(-y'',x'',0)+2(xy'-x'y)' F$ is parallel to $\pa_t F$. Equivalently, there is a function $k=k(t,z)$ such that 
$$(-y'',x'',0)+2(xy'-x'y)' F=-k[(x',y',0)-2(xx'+yy')F]=-k(x^2+y^2+|z|^2) \pa_t F.$$
Solving the above equation gives $k=k(t)=x'y''-x''y'$, which is exactly the relative curvature of the plane curve, i.e.  $(x'',y'')=k(t)(-y',x')$. Therefore, putting these facts together we obtain
\begin{equation}\label{principal_curvature:simple}
\mu=\lambda-k(t)(x^2+y^2+|z|^2).
\end{equation}

Finally, we have already known that $Q$ is umbilical. Then it follows that  $\lambda-\mu$ goes to zero for all $(t,z)$ with $t^2+|z|^2 \to \infty$. Subsequently, for each fixed $t$, letting $|z| \to \infty$ in \eqref{principal_curvature:simple} forces $k(t)=0$ for all $t \in \RR$, and hence $(x(t),y(t)), t \in \RR$ is a straight line. Therefore, the hypersurface $U$ is a flat Euclidean space $\RR^n$, and then $(M,g)$ must be a round sphere.
\end{proof}

Finally, we give an evidence to support our conjecture for surfaces:

\begin{theorem}\label{thm:rigidity_analytic}
 Let $M^2$ be a closed real analytic surface in $\RR^3$ with induced metric $g$. Suppose there is a point $Q\in M$ such that the Green function of $g$ has the form $G(P,Q)=-\frac{1}{2\pi}\log\|P-Q\|+C$ for any  $P\in M$ and some $C\in \RR$. Assume additionally that the surface is $\Sp^1$-invariant about $Q$. Then $M$ is a round sphere.   
\end{theorem}

\begin{proof}
    We follow the same notation as in the proof to Theorem \ref{thm:rigidity}(1). Then the equation \eqref{eqn:Green_surface}  becomes 
    $$2H\frac{\eta}{\rho}+4\Big(\frac{\eta}{\rho}\Big)^2=-c=:-4c_0.$$
    Now by our $\Sp^1$-invariant assumption, we can write $f(x)$ as $f(|x|^2)=:f(t)$. Then $\rho=t+f^2(t)$, and
    $$\eta=\frac{f(t)-2tf'(t)}{\sqrt{1+4tf'(t)^2}}.$$
    A direct computation shows
    $$H=\frac{4(f'(t)+tf''(t)+2tf'(t)^3)}{(1+4tf'(t)^2)^{\frac{3}{2}}}.$$
    Then the above equation becomes
    $$\frac{2(f'+tf''+2tf'^3)(f-2tf')}{(t+f^2)(1+4tf'^2)^2}+\frac{(f-2tf')^2}{(t+f^2)^2(1+4tf'^2)}=-c_0.$$
    Clearing the denominator, we get
    \begin{align*}
        &2(f'+tf''+2tf'^3)(f-2tf')(t+f^2)+(f-2tf')^2(1+4tf'^2)\\
        =&-c_0(t+f^2)^2(1+4tf'^2)^2.
    \end{align*}

    Now the real analytic assumption enables us to assume $f(t)=\sum_{n=1}^\infty a_n t^n$. 
    Denote
    \begin{align*}
        t+f^2&=t\sum_{n=0}^\infty A_n t^n;\\
        f-2tf'&=-t\sum_{n=0}^\infty B_n t^n;\\
        f'+tf''+2tf'^3&=\sum_{n=0}^\infty C_n t^n;\\
        1+4tf'^2&=\sum_{n=0}^\infty D_n t^n.
    \end{align*}
    Then we have
    $$A_0=1, \quad A_n=\sum_{\substack{k+l=n+1\\ k,l\geq 1}}a_ka_l,\ (n\geq 1);\quad  B_n=(2n+1)a_{n+1},\ (n\geq 0);$$
    $$C_0=a_1,\quad C_n=(n+1)^2a_{n+1}+\sum_{\substack{k+l+p=n-1\\ k,l,p\geq 0}}2(k+1)(l+1)(p+1)a_{k+1}a_{l+1}a_{p+1},\ (n\geq 1)$$
    and
    $$D_0=1,\quad D_n=\sum_{\substack{k+l=n-1\\ k,l\geq 0}}(k+1)(l+1)a_{k+1}a_{l+1},\ (n\geq 1).$$
    Observe that the expressions of both $A_n$ and $D_n$ only involve $a_1,\cdots, a_n$, while $B_n$ and $C_n$ contain terms involving $a_{n+1}$ as well, with coefficients $2n+1$ and $(n+1)^2$ respectively.

    Now
    \begin{align*}
        \frac{2}{t^2}(f'+tf''+2tf'^3)(f-2tf')(t+f^2)&=-2\sum_{n\geq 0}\Big(\sum_{\substack{n_1+n_2+n_3=n\\ n_i\geq 0}}A_{n_1}B_{n_2}C_{n_3}\Big)t^n,\\
        \frac{1}{t^2}(f-2tf')^2(1+4tf'^2)&=\sum_{n\geq 0}\Big(\sum_{\substack{n_1+n_2+n_3=n\\ n_i\geq 0}}B_{n_1}B_{n_2}D_{n_3}\Big)t^n,\\
        \frac{-c_0}{t^2}(t+f^2)^2(1+4tf'^2)^2&=-c_0\sum_{n\geq 0}\Big(\sum_{\substack{n_1+n_2+n_3+n_4=n\\ n_i\geq 0}}A_{n_1}A_{n_2}D_{n_3}D_{n_4}\Big)t^n.
    \end{align*}
    
     Based on these formulae, we arrive at a recursive formula: $a_1^2=c_0$ and
    $$\Big[-2A_0\Big((2n+1)C_0+B_0(n+1)^2\Big)+2(2n+1)B_0D_0\Big]a_{n+1}=\Phi_n(a_1,\cdots, a_n),$$
    where $\Phi_n(a_1,\cdots,a_n)$ is a polynomial of $a_1,\cdots,a_n$. 
    That is
    $$-2(n+1)^2a_1a_{n+1}=\Phi_n(a_1,\cdots, a_n).$$
    Since $c_0>0$, we can recursively determine all the coefficients. So we have the uniqueness. On the other hand, we can now verify that locally a round sphere with a graph $(x,f(x))$ is a solution, where $f(x)=\frac{1}{\sqrt{c}}(1-\sqrt{1-c|x|^2})$ for $x \in B_{\frac{1}{\sqrt{c}}}(0)$. So it is  the only solution.    
\end{proof}

\appendix

\section{Alternative proof of Theorem \ref{thm:rigidity}(2) when $n\geq 5$}

To start, we need the following elementary lemma comparing extrinsic and intrinsic distance functions. This might be known to experts, though it is difficult for us to identify a proper reference at present. 
\begin{lemma}\label{lem:dist}
    Suppose $P$ is sufficiently close to $Q$, denote by $\gamma$ the unique unit speed minimizing geodesic from $Q$ to $P$ and $\rho(P):=\|P-Q\|^2$. Let $r$ be the intrinsic distance from $Q$ along $\gamma$, then we have
    $$\rho=r^2-\frac{1}{12}II\big(\gamma'(0),\gamma'(0)\big)^2r^4+o(r^4),$$
    along $\gamma$, where $II$ is the second fundamental form.
\end{lemma}

\begin{proof}
    As before, we assume that $Q=0\in\RR^{n+1}$ and  locally $M$ is given as the graph of a function $f:B_\delta(0)\to \RR$ with $f(0)=0$ and $\nabla f(0)=0$.  Then under the local coordinate chart $x=(x_1,\cdots, x_n)$, we have $g_{\alpha\beta}=\delta_{\alpha\beta}+f_\alpha f_\beta$, and $\Gamma_{\beta\gamma}^\alpha=\frac{f_{\beta\gamma}f_\alpha}{1+|\nabla f|^2}$. Let $\gamma(t)=(x(t), f(x(t)))$ be the unit speed geodesic from $Q$ to $P$. Then we have $x(0)=0$, $\sum_\alpha (x'_\alpha)^2+(\sum_\alpha f_\alpha x'_\alpha)^2=1$, and the system of differential equations for the geodesic: 
    $$x''_\alpha+\frac{f_\alpha}{1+|\nabla f|^2}(\sum_{\beta,\gamma}f_{\beta\gamma}x'_\beta x'_\gamma)=0, \quad \forall~ \alpha.$$
    We also assume that $(f_{\alpha\beta}(0))=\mathrm{diag}(\kappa_1,\cdots,\kappa_n)$, then these $\kappa_\alpha$'s are just the principal curvatures of $M$ at $Q$.
    Now $\rho(t):=\sum_\alpha x_\alpha^2(t)+f^2(x(t))$, so we get
    \begin{align*}
        \rho'&=2\sum_\alpha x_\alpha x'_\alpha+2f\sum_\alpha f_\alpha x'_\alpha
    \end{align*}
    and hence $\rho'(0)=0$.
    Next, we have
    \begin{align*}
        \rho''&=2\sum_\alpha x_\alpha x''_\alpha+2f\sum_\alpha f_\alpha x''_\alpha+2\sum_\alpha ( x'_\alpha)^2+2(\sum_\alpha f_\alpha x'_\alpha)^2+2f\sum_{\alpha,\beta}f_{\alpha\beta} x'_\alpha x'_\beta\\
        &=-2\sum_\alpha (x_\alpha+ff_\alpha)\frac{f_\alpha}{1+|\nabla f|^2}(\sum_{\beta,\gamma}f_{\beta\gamma}x'_\beta x'_\gamma)+2+2f\sum_{\beta,\gamma}f_{\beta\gamma} x'_\beta x'_\gamma\\
        &=\frac{2(f-x\cdot\nabla f)}{1+|\nabla f|^2}(\sum_{\beta,\gamma}f_{\beta\gamma}x'_\beta x'_\gamma)+2.
    \end{align*}
    This yields $\rho''(0)=2$. Taking derivative with respect to $t$ again, we have
    \begin{align*}
        \rho^{(3)}&=\frac{2(f-x\cdot\nabla f)}{1+|\nabla f|^2}\Big[\sum_{\beta,\gamma,\mu}f_{\beta\gamma\mu}x'_\beta x'_\gamma x'_\mu+2\sum_{\beta,\gamma}f_{\beta\gamma}x'_\beta x''_\gamma\Big]\\
        &\quad +(\sum_{\beta,\gamma}f_{\beta\gamma}x'_\beta x'_\gamma)\Big[\frac{2(\sum_\alpha f_\alpha  x'_\alpha-\sum_\alpha  x'_\alpha f_\alpha-\sum_{\alpha,\mu}f_{\alpha\mu}x_\alpha x'_\mu)}{1+|\nabla f|^2}\\
        &\quad -\frac{4(f-x\cdot\nabla f)(\sum_{\alpha,\beta}f_{\alpha\beta}f_\alpha x'_\beta)}{(1+|\nabla f|^2)^2}\Big]\\
        &=\frac{2(f-x\cdot\nabla f)}{1+|\nabla f|^2}\Big[\sum_{\beta,\gamma,\mu}f_{\beta\gamma\mu}x'_\beta x'_\gamma x'_\mu-2(\sum_{\beta,\gamma}f_{\beta\gamma}x'_\beta f_\gamma)\frac{\sum_{\mu,\nu}f_{\mu\nu} x'_\mu x'_\nu}{1+|\nabla f|^2}\Big]\\
        &\quad-2(\sum_{\beta,\gamma}f_{\beta\gamma}x'_\beta x'_\gamma)\Big[\frac{\sum_{\alpha,\mu}f_{\alpha\mu}x_\alpha x'_\mu}{1+|\nabla f|^2}+\frac{2(f-x\cdot\nabla f)(\sum_{\alpha,\beta}f_{\alpha\beta}f_\alpha x'_\beta)}{(1+|\nabla f|^2)^2}]\\
        &=\frac{2(f-x\cdot\nabla f)}{1+|\nabla f|^2}\Big[\sum_{\beta,\gamma,\mu}f_{\beta\gamma\mu}x'_\beta x'_\gamma x'_\mu-4(\sum_{\beta,\gamma}f_{\beta\gamma}x'_\beta f_\gamma)\frac{\sum_{\mu,\nu}f_{\mu\nu} x'_\mu x'_\nu}{1+|\nabla f|^2}\Big]\\
        &\quad-2(\sum_{\beta,\gamma}f_{\beta\gamma}x'_\beta x'_\gamma)\frac{\sum_{\alpha,\mu}f_{\alpha\mu}x_\alpha x'_\mu}{1+|\nabla f|^2}.
    \end{align*}
    This implies $\rho^{(3)}(0)=0$. Finally, we have
    \begin{align*}
        \rho^{(4)}(0)&=-2\Big(\sum_{\beta,\gamma}f_{\beta\gamma}(0)x'_\beta(0)x'_\gamma(0)\Big)\frac{\sum_{\alpha,\mu}f_{\alpha\mu}(0) x'_\alpha(0) x'_\mu(0)}{1+|\nabla f|^2(0)}\\
        &=-2\Big(II(\gamma'(0),\gamma'(0))\Big)^2.
    \end{align*}
    Then by the Taylor formula, we get the expansion of $\rho$ with respect to $r$.
\end{proof}

\begin{proof}[Alternative proof of Theorem \ref{thm:rigidity} (2) when $n\geq 5$] By Lemma \ref{lem:dist}, we obtain the following asymptotic expansion of $G(P,Q)$ for the conformal Laplacian near the diagonal $P=Q$:
\begin{align*}
    G(P,Q)&= c_{n,1}\|P-Q\|^{2-n}=c_{n,1}\rho^{1-\frac{n}{2}}\\
    &=c_{n,1}r^{2-n}\Big(1-\frac{1}{12}II\big(\gamma'(0),\gamma'(0)\big)^2 r^2+o(r^2)\Big)^{1-\frac{n}{2}}\\
    &=c_{n,1} r^{2-n}+\frac{c_{n,1}(n-2)}{24}II\big(\gamma'(0),\gamma'(0)\big)^2 r^{4-n}+o(r^{4-n}).
\end{align*}

Now we need a result due to Parker-Rosenberg
\cite[Theorem 2.2 and (4.4)]{ParkerRosenberg}:
$$G(P,Q)=(4\pi)^{-\frac{n}{2}}\Big[\sum_{k=0}^{[\frac{n-3}{2}]}a_k(P,Q)(\frac{r}{2})^{2-n+2k}\Gamma(\frac{n}{2}-k-1)\Big]+\mathrm{Higher~~order~~terms}$$
with 
$$a_0(P,Q)=(\det g)^{-1/4}=1+\frac{1}{12}\mathrm{Ric}(\gamma'(0),\gamma'(0))r^2+o(r^2)$$
and
$$a_1(P,Q)=\frac{4-n}{12(n-1)}R_g(Q)+o(1).$$
Comparing with our expansion above, we have
$$c_{n,1}=(4\pi)^{-\frac{n}{2}}2^{n-2}\Gamma(\frac{n}{2}-1)$$
and 
$$\frac{c_{n,1}(n-2)}{24}II\big(\gamma'(0),\gamma'(0)\big)^2=(4\pi)^{-\frac{n}{2}}2^{n-2}\Gamma(\frac{n}{2}-1)\Big[\frac{1}{12}\mathrm{Ric}(\gamma'(0),\gamma'(0))-\frac{R_g(Q)}{24(n-1)}\Big],$$
which in turn imply
$$(n-2)II(\gamma'(0),\gamma'(0))^2=2\mathrm{Ric}(\gamma'(0),\gamma'(0))-\frac{R_g(Q)}{n-1}.$$
Taking trace together with the fact that $R_g=H^2-|II|^2$, we obtain $|II(Q)|^2=\frac{R_g(Q)}{n-1}$, i.e. $n|II(Q)|^2=H^2(Q)$, and hence $Q$ is an umbilical point.

The remaining part of proof is the same as \S \ref{subsec:conf_lap}.
\end{proof}

\end{document}